\documentclass[a4paper,10pt]{amsart}
\usepackage{amscd}
\usepackage{amsfonts}
\usepackage{amsmath}
\usepackage{amssymb}
\usepackage{cases}
\usepackage{graphicx,colortbl}
\usepackage{latexsym}
\usepackage{mathrsfs}
\usepackage[all]{xy}
\usepackage{tikz}

\begin{document}
\input xy
\xyoption{all}

\newcommand{\crl}{{\mathscr L}}
\newcommand{\bbz}{\mathbb{Z}}
\renewcommand{\mod}{\operatorname{mod}\nolimits}
\newcommand{\proj}{\operatorname{proj}\nolimits}
\newcommand{\rad}{\operatorname{rad}\nolimits}
\newcommand{\soc}{\operatorname{soc}\nolimits}
\newcommand{\ind}{\operatorname{ind}\nolimits}
\newcommand{\Top}{\operatorname{top}\nolimits}
\newcommand{\ann}{\operatorname{Ann}\nolimits}
\newcommand{\id}{\operatorname{id}\nolimits}
\newcommand{\Mod}{\operatorname{Mod}\nolimits}
\newcommand{\End}{\operatorname{End}\nolimits}
\newcommand{\Ob}{\operatorname{Ob}\nolimits}
\newcommand{\Ht}{\operatorname{Ht}\nolimits}
\newcommand{\cone}{\operatorname{cone}\nolimits}
\newcommand{\rep}{\operatorname{rep}\nolimits}
\newcommand{\Ext}{\operatorname{Ext}\nolimits}
\newcommand{\Hom}{\operatorname{Hom}\nolimits}
\renewcommand{\Im}{\operatorname{Im}\nolimits}
\newcommand{\Ker}{\operatorname{Ker}\nolimits}
\newcommand{\Coker}{\operatorname{Coker}\nolimits}
\renewcommand{\dim}{\operatorname{dim}\nolimits}
\newcommand{\Ab}{{\operatorname{Ab}\nolimits}}
\newcommand{\Coim}{{\operatorname{Coim}\nolimits}}
\newcommand{\pd}{\operatorname{pd}\nolimits}
\newcommand{\sdim}{\operatorname{sdim}\nolimits}
\newcommand{\add}{\operatorname{add}\nolimits}
\newcommand{\cc}{{\mathcal C}}
\newcommand{\crc}{{\mathscr C}}
\newcommand{\crd}{{\mathscr D}}
\newcommand{\co}{{\mathcal O}}
\newcommand{\cv}{{\mathcal V}}
\newcommand{\cs}{{\mathcal S}}

\newcommand{\coh}{\operatorname{coh}\nolimits}
\newcommand{\vect}{\operatorname{vect}\nolimits}
\newcommand{\can}{\operatorname{can}\nolimits}
\newcommand{\rk}{\operatorname{rk}\nolimits}
\newcommand{\st}{[1]}
\newcommand{\X}{\mathbb{X}}
\newcommand{\Z}{\mathbb{Z}}
\newcommand{\cp}{\mathcal{P}}
\newcommand{\ci}{\mathcal{I}}
\newcommand{\bbl}{\mathbb{L}}
\newcommand{\ul}{\underline}
\newcommand{\vx}{\vec{x}}
\newcommand{\vy}{\vec{y}}
\newcommand{\vz}{\vec{z}}
\newcommand{\vc}{\vec{c}}
\newcommand{\vw}{\vec{\omega}}
\newcommand{\bx}{\bar{x}}
\newcommand{\cub}{\operatorname{cub}\nolimits}

\newtheorem{theorem}{Theorem}[section]
\newtheorem{acknowledgement}[theorem]{Acknowledgement}
\newtheorem{algorithm}[theorem]{Algorithm}
\newtheorem{axiom}[theorem]{Axiom}
\newtheorem{case}[theorem]{Case}
\newtheorem{claim}[theorem]{Claim}
\newtheorem{conclusion}[theorem]{Conclusion}
\newtheorem{condition}[theorem]{Condition}
\newtheorem{conjecture}[theorem]{Conjecture}
\newtheorem{corollary}[theorem]{Corollary}
\newtheorem{criterion}[theorem]{Criterion}
\newtheorem{definition}[theorem]{Definition}
\newtheorem{example}[theorem]{Example}
\newtheorem{exercise}[theorem]{Exercise}
\newtheorem{lemma}[theorem]{Lemma}
\newtheorem{notation}[theorem]{Notation}
\newtheorem{proposition}[theorem]{Proposition}
\newtheorem{remark}[theorem]{Remark}
\newtheorem{solution}[theorem]{Solution}
\newtheorem{summary}[theorem]{Summary}
\newtheorem*{thma}{Theorem}

\newtheorem*{Riemann-Roch Formula}{Riemann-Roch Formula}

\newtheorem*{question}{Question}

\numberwithin{equation}{section}

%%%%%%%%%%%%%%%%%%%%%%%%%%
\def\lra{\longrightarrow}
\def\wtd{\widetilde}
\def\cT{{\mathcal T}}
\def\ca{\mathcal A}

\def\bbX{{\mathbb X}}
\def\bbF{{\mathbb F}}
\def\bbH{{\mathbb H}}
\def\bbZ{{\mathbb Z}}
\def\bbN{{\mathbb N}}
\def\bbL{{\mathbb L}}
\def\bbK{{\mathbb K}}
\def\bbQ{{\mathbb Q}}
\def\bbC{{\mathbb C}}
\def\bbP{{\mathbb P}}
\def\bbO{{\mathbb O}}
\def\bbT{{\mathbb T}}
\def\lz{\lambda}
\def\rmx{{\rm x}}
\def\dz{\delta} \def\bfk{{\mathbf k}}

%%%%%%%%%%%%%%%%%%%%%%%%%%%%%%%%%%%%%%%%%%%%%%%%%%%%%%%%%%%%%%%%%%%%%%%%%%%%%%%%%%%%%%%%%%%%%%%%%%%%%%%
%%%%%%%%%%%%%%%%%%%%%%%%%%%%%%%%%%%%%%%%%%%%%%%%%%%%%%%%%%%%%%%%%%%%%%%%%%%%%%%%%%%%%%%%%%%%%%%%%%%%%%%

\title[Hall Polynomials for Weighted projective lines]
{Hall Polynomials for Weighted projective lines}

\author[Jiayi Chen, Bangming Deng and Shiquan Ruan]{Jiayi Chen, Bangming Deng and Shiquan Ruan}
\address{School of Mathematical Sciences, Xiamen University, Xiamen 361005, P.R.China}
\email{jiayichen.xmu@foxmail.com}
\address{Department of Mathematical Sciences, Tsinghua University, Beijing 100084, P.R.China}
\email{bmdeng@tsinghua.edu.cn}
\address{School of Mathematical Sciences, Xiamen University, Xiamen 361005, P.R.China}
\email{sqruan@xmu.edu.cn}

\subjclass[2020]{17B37, 16G20, 18G60} \keywords{Hall polynomial; Hall algebra; coherent sheaf; extension bundle; Weighted projective
line.}

\dedicatory{Dedicated to Professor Claus Michael Ringel on the occasion of his 80th birthday}

\begin{abstract}
This paper deals with the triangle singularity defined by the \linebreak equation
$f=X_1^{p_1}+X_2^{p_2}+X_3^{p_3}$ for weight triple $(p_1,p_2,p_3)$, as well as
the category of coherent sheaves over the weighted projective line $\bbX$ defined by $f$.
We calculate Hall polynomials associated to extensions bundles, line bundles
and torsion sheaves over $\bbX$. By using derived equivalence,
this provides a unified conceptual method for calculating
Hall polynomials for representations of tame quivers obtained by
Sz\'ant\'o and Sz\"oll\H{o}si [J. Pure Appl. Alg. {\bf 228} (2024)].

\end{abstract}

\date{\today}

\maketitle

%%%%%%%%%%%%%%%%%%%%%%%%%%%%%%%%%%%%%%%%%%%%%%%%%%%%%%%%%%%%%%%%%%%%%%%%%%%%%%%%%%%%%%%%%%%%%%%%%%%%%%%%%%%%%%%%%%%%%%%%%%%%%%%%%%%%%%%%%%%%%%%%%%%%%%
\section{Introduction}
%%%%%%%%%%%%%%%%%%%%%%%%%%%%%%%%%%%%%%%%%%%%%%%%%%%%%%%%%%%%%%%%%%%%%%%%%%%%%%%%%%%%%%%%%%%%%%%%%%%%%%%%%%%%%%%%%%%%%%%%%%%%%%%%%%%%%%%%%%%%%%%%%%%%%%

The classical Hall algebra was first studied by Steinitz \cite{St}
and later by Hall \cite{Ha} in the context of finite abelian $p$-groups. It turns out
that this algebra is isomorphic to the ring of symmetric functions and plays an
important role in the study of representations of symmetric groups and general linear
groups; see \cite{Mac}. In 1990, Ringel \cite{Rin1} introduced the Hall algebra
$H(R)$ of a finitary ring $R$, which is the
free abelian group with basis the set of isomorphism classes of finite
$R$-modules, and the structure constants are given by the so-called Hall numbers
$$F_{M,N}^L:=|\{X\subseteq L\mid X\cong N,\,L/X\cong M\}|.$$
 where $L,M,N$ are finite $R$-modules. Ringel then proved that the Hall algebra of
a finite dimensional representation-finite hereditary algebra over a finite field $\bfk$
is isomorphic to the positive part of the corresponding quantum group.
By defining a bialgebra structure on the Hall algebra $H(R)$ of an arbitrary finite
dimensional hereditary algebra $R$, Green \cite{Gr} proved that the composition subalgebra
of $H(R)$ generated by simple modules gives a realization of the positive part of the
corresponding quantum group.

In \cite{Rin0}, Ringel showed that Hall numbers for representation-directed algebras
over finite fields $\bfk$ are actually integer polynomials in the cardinality $|\bfk|$,
called Hall polynomials, and he further calculated Hall polynomials for three indecomposable
modules over a representation-finite hereditary algebra in \cite{Rin2}. Since then,
much subsequent work is concerned with the existence of Hall polynomials and explicit
formulas of certain Hall polynomials. For example, the existence
of Hall polynomials for representations of tame quivers was studied in
\cite{Rin3, Guo, Zh-GL, Hub, DR17, DH}, and Hall polynomials for some special classes
of representations of tame quivers were calculated in \cite{Zhang, Sz0,Sz1,SS1,SS2}.

We remark that Hall polynomials have significant applications in the study of cluster
algebras, as well as for counting rational points of quiver moduli
and for constructing canonical bases of affine quantum groups;
see, e.g., \cite{CC, CK, Rein, CR, Sz1, DDX, XXuZh, XXu}.

The notions of weighted projective lines and their coherent sheaf categories were introduced by
Geigle and Lenzing \cite{GL} in order to provide geometric counterparts to canonical
algebras and their module categories studied by Ringel \cite{R84}.
The Hall algebra of a weighted projective line was studied by Schiffmann \cite{Sch1} and
further by Dou--Jiang--Xiao in connection with quantum loop algebras. The existence of Hall
polynomials for coherent sheaves over a weighted projective line of domestic type was proved
in \cite{DR17}.

It is well-known that the category of representations of a tame quiver and that of
coherent sheaves over the associated weighted projective line bear many similarities.
For example, they both are hereditary abelian categories and have equivalent bounded
derived categories. It seems to us that it is sometimes convenient
to work with weighted projective lines instead of tame quivers since the category
of coherent sheaves admits a natural torsion pair consisting of
vector bundles and torsion sheaves, respectively.

In this paper we mainly consider weighted projective lines $\X$ of
type $(p_1,p_2,p_3)$ over a finite field $\bfk$ of $q$ elements and aim
to calculate Hall polynomials associated with line bundles,
extension bundles (i.e., indecomposable bundles of rank $2$) and torsion sheaves over $\X$.
As applications, we obtain some Hall polynomials for representations of tame quivers.

We first consider an easy case where the Hall numbers only involve line
bundles and torsion sheaves over an arbitrary weighted projective line $\X$.
Next we treat the main case where $\X$ is a weighted projective line of
type $(p_1,p_2,p_3)$ and an extension bundle arises as an extension of
two line bundles. Namely, assume $L_1,L_2$ are line bundles and $E=E_L\langle \vx\rangle$
is an extension bundle defined by a line bundle $L$ and an element $\vx\in\bbl$.
Then the Hall number $F^E_{L_1 L_2}$ has a close relation with the
automorphism groups of torsion sheaves $S$ which fit into the following
commutative diagram with exact row and column sequences:
\[
\begin{tikzpicture}
\node (-10) at (-1.5,0) {$0$};
\node (-11) at (-1.5,1) {$0$};
\node (40) at (4.5,0) {$0$.};
\node (41) at (4.5,1) {$0$};
\node (1-1) at (1.5,-1) {$0$};
\node (2-1) at (3,-1) {$0$};
\node (13) at (1.5,3) {$0$};
\node (23) at (3,3) {$0$};
\node (0) at (0,0) {$L_1$};
\node (1) at (1.5,0) {$L(\vx)$};
\node (2) at (3,0){$S$};
\node (01) at (0,1) {$L_1$};
\node (11) at (1.5,1) {$W$};
\node (21) at (3,1){$L_2$};
\node (02) at (0,2) {};
\node (12) at (1.5,2) {$L(\vw)$};
\node (22) at (3,2){$L(\vw)$};	
\draw[->] (-10) --node[above ]{} (0);
\draw[->] (-11) --node[above ]{} (01);
\draw[->] (2) --node[above ]{} (40);
\draw[->] (21) --node[above ]{} (41);
\draw[->] (1) --node[above ]{} (1-1);
\draw[->] (2) --node[above ]{} (2-1);
\draw[->] (13) --node[above ]{} (12);
\draw[->] (23) --node[above ]{} (22);

\draw[->] (01) --node[above ]{} (11);
\draw[->] (11) --node[above ]{} (21);
\draw[->] (12) --node[above ]{} (11);
\draw[->] (11) --node[above ]{} (1);	
\draw[->] (0) --node[above ]{} (1);
\draw[->] (1) --node[above ]{} (2);
\draw[double] (12) --node[above ]{} (22);
\draw[double] (0) --node[above ]{} (01);
\draw[->] (21) --node[above ]{} (2);
\draw[->] (22) --node[above ]{} (21);	
\end{tikzpicture} \]
By repeatedly making use of Green's formula, we will prove that
\[F_{L_2L_1}^E=f_{n}(q) {\;\;\text{with}\;\;} n=\dim \Hom(L_1,L_2)-\dim \Ext^1(L_1,L_2),\]
 where $f_{n}(q)$ is defined in \eqref{def-f(q)}; see \cite{SS2}.
Furthermore, we describe Hall polynomials associated with two extension bundles and
an indecomposable torsion sheaf. This requires a bit more detailed analysis and
calculations. Finally, with the help of the embeddings from
both Hall algebra of a tame quiver and that of the associated weighted projective line $\X$ into
the corresponding derived Hall algebra, we obtain Hall polynomials for tame quivers
 which are consistent with the results in \cite{SS2}.

Our calculation is based on a detail description of extension bundles on $\X$ due to
Kussin--Lenzing--Meltzer \cite{KLM}, see also \cite{DR25}.
In comparison with the proofs given in \cite{SS2}, our approach seems to be simpler and more
conceptual since we have a unified way to describe the homomorphisms
among extension bundles, line bundles and torsion sheaves. Moreover, our
results are valid for all weighted projective lines of three weights, which may
be of domestic, tubular or wild type.

The paper is organized as follows. Section 2 recalls the definitions of
coherent sheaves on a weighted projective line and Hall algebra
of a finitary category and states some basic facts. The main calculations of
Hall polynomials are given in Sections 3 and 4, and we obtain Hall polynomials
involving line bundles, extension bundles, and torsion sheaves.
In Section 5, we consider Hall polynomials in bounded derived
categories of weighted projective lines which give rise to Hall polynomials
for representations of tame quivers.

\bigskip

%%%%%%%%%%%%%%%%%%%%%%%%%%%%%%%%%%%%%%%%%%%%%%%%%%%%%%%%%%%%%%%%%%%%%%%%%%%%%%%%%%%%%%%%%%%%%%%%%%%%%%%%%%%%%%%%%%%%%%%%%%%%%%%%%%%%%%%%%%%%%%%%%%%%%%
\section{Preliminaries and some preparatory results}
%%%%%%%%%%%%%%%%%%%%%%%%%%%%%%%%%%%%%%%%%%%%%%%%%%%%%%%%%%%%%%%%%%%%%%%%%%%%%%%%%%%%%%%%%%%%%%%%%%%%%%%%%%%%%%%%%%%%%%%%%%%%%%%%%%%%%%%%%%%%%%%%%%%%%%

In this section we review the category of coherent sheaves over a weighted projective
line and its basic properties, and introduce the Ringel--Hall algebra of a finitary
category and, in particular, Green's formula in the hereditary case. Orthogonal exceptional
pair associated to an exceptional object in an abelian category is also defined.

\subsection{Coherent sheaves on weighted projective lines}

Let $\mathbf{k}$ be an arbitrary field. A \emph{weighted projective line}
$\X$ over $\mathbf{k}$ is specified by giving a \emph{weight
sequence} ${\bf p}=(p_{1},p_{2},\ldots, p_{t})$ of positive
integers, and a collection
${\boldsymbol\lambda}=(\lambda_{1},\lambda_{2},\ldots, \lambda_{t}) $ of
distinct points in the projective line $\bbP^{1}(\mathbf{k})$ which can be
normalized as $\lambda_{1}=\infty, \lambda_{2}=0, \lambda_{3}=1$.
More precisely, let $\bbL=\bbL(\bf p)$ be the rank one abelian group with
generators $\vec{x}_{1}, \vec{x}_{2}, \ldots, \vec{x}_{t}$ and the
relations
\[ p_{1}\vx_1=p_{2}\vx_2=\cdots=p_{t}\vx_t=:\vec{c},\]
where $\vec{c}$ is called the \emph{canonical element} of
$\mathbb{L}$. Denote by $S$ the commutative algebra
\[S=\mathbf{k}[X_{1},X_{2},\cdots,
X_{t}]/{\frak a}:= \mathbf{k}[{\rm x}_{1},{\rm x}_{2}, \ldots, {\rm x}_{t}],\]
where ${\frak a}=(f_{3},\ldots,f_{t})$ is the ideal generated by
$f_{i}=X_{i}^{p_{i}}-X_{2}^{p_{2}}+\lambda_{i}X_{1}^{p_{1}},
i=3,\ldots, t$. Put $I=\{1,2,\ldots,t\}$. Then $S$ is $\mathbb{L}$-graded by setting
$$\mbox{deg}(\rmx_{i})=\vx_i\; \text{ for each $i\in I$.}$$
Moreover, each element $\vx\in\bbL$ has the normal form $\vx=\sum_{i\in I} l_i\vx_i+l\vc$
with $0\leq l_i<p_i$ and $l\in\bbZ$. We denote by $\bbL_+$ the positive cone of $\bbL$
which consists of those $\vx$ with $l\geq 0$. Finally, the weighted projective line
associated with $\bf p$ and $\boldsymbol\lz$ is defined to be $\bbX=\rm{Proj}^{\bbL}S$.

According to \cite{GL}, the set of nonzero prime homogeneous
elements in $S$ is partitioned into two sets: the
exceptional primes $\rmx_1,\ldots, \rmx_t$ and the ordinary primes
$f(\rmx_{1}^{p_{1}}, \rmx_{2}^{p_{2}})$, where $f(T_1, T_{2})$ is a prime
homogeneous polynomial in $\mathbf{k}[T_1,T_2]$ which are distinct from $T_1, T_2$ and
$T_2-\lz_iT_1$ for $i\in I$. The exceptional primes
correspond to the points $\lz_1, \ldots, \lz_t$, called
exceptional points and denoted by $x_1,\ldots,x_t$, respectively, while the ordinary
primes correspond to the remaining closed points of $\bbP^{1}(\mathbf{k})$, called ordinary
points. For convenience, we denote by ${\mathbb H}_k$ the set of ordinary points.
For each $z\in {\mathbb H}_\mathbf{k}$, its degree $\deg\,z$ is defined to be the degree
of the corresponding prime homogeneous polynomial.

The category of coherent sheaves on $\bbX$ can be defined as the
quotient of the category ${\rm mod}^{\mathbb{L}}(S)$ of finitely generated
$\mathbb{L}$-graded $S$-modules over the Serre subcategory $\mbox{mod}_{0}^{\mathbb{L}}(S)$
of finite length modules, that is,
$$\coh\bbX:={\rm mod}^{\mathbb{L}}(S)/\mbox{mod}_{0}^{\mathbb{L}}(S).$$
The free module $S$ gives the structure sheaf $\co$. Each line
bundle is given by the grading shift $\co(\vec{x})$ for a uniquely
determined element $\vec{x}\in \mathbb{L}$, and there is an
isomorphism
\[\Hom(\co(\vec{x}), \co(\vec{y}))\cong S_{\vec{y}-\vec{x}}.\]
Moreover, $\coh\bbX$ is a hereditary abelian category with Serre
duality of the form
\[D\Ext^1(X, Y)\cong\Hom(Y, X(\vec{\omega})),\]
where $D=\Hom_\mathbf{k}(-,\mathbf{k})$, and $\vec{\omega}:=(t-2)\vec{c}-\sum_{i\in I}\vec{x}_{i}$ is
called the \emph{dualizing element} of $\mathbb{L}$. This implies
the existence of almost split sequences in $\coh\X$ with the
Auslander--Reiten translation $\tau$ given by the grading shift with
$\vec{\omega}$.

Recall that $\coh\X$ admits a splitting torsion pair $({\rm coh}_0\X,\vect{\X})$,
where ${\rm coh}_0\X$ and $\vect{\X}$ are full subcategories
of torsion sheaves and vector bundles, respectively. Moreover, ${\rm coh}_0\X$
decomposes as a direct product of orthogonal tubes
$${\rm coh}_0\X=\prod_{i\in I}\mathcal{T}_i\times \prod_{z\in\bbH_\mathbf{k}}\mathcal{T}_z,$$
 where each $\mathcal{T}_z$ is a homogeneous tube for $z\in \bbH_\mathbf{k}$,
which is equivalent to the category of nilpotent representations of the
Jordan quiver over the residue field $\mathbf{k}_z$, while each $\mathcal{T}_i$
is a non-homogeneous tube for $i\in I$, which is equivalent to the
category of nilpotent representations of the cyclic quiver with $p_i$ vertices.
For each $i\in I$, let $S_{i,j}=S_{i,j}^{(1)}$ be the simple objects in $\mathcal{T}_i$,
where $0\leq j\leq p_i-1$, such that
$$\tau S_{i,j}=S_{i,j-1}.$$
Further, for $n\geq 1$, we denote by $S_{i,{j}}^{(n)}$
the unique indecomposable object in $\mathcal{T}_i$ whose top object is $S_{i,j}$ and length equals to $n$.
While for each $z\in \bbH_\mathbf{k}$, let $S_z=S_{z}^{(1)}$ be the unique simple object
in $\mathcal{T}_z$, and denote by $S_{z}^{(n)}$ the indecomposable object in $\mathcal{T}_z$
of length $n$.

It is known that the Grothendieck group $K_0(\X)$ of $\coh\X$ is a free abelian group
with a basis $[\co(\vec{x})]$ with $0\leq\vec{x}\leq \vec{c}$, where $[X]$ denotes
the class of an object $X\in \coh\X$ in $K_0(\X)$. Set $\delta=[L(\vc)]-[L]$.
The Euler form on $K_0(\X)$ is given by
\[\langle [X],[Y]\rangle=\dim_{k}\Hom(X,Y)-\dim_{k}\Ext^{1}(X,Y)\]
for any $X,Y\in\coh\X$. For simplicity, we sometimes write $\langle X,Y\rangle=\langle [X],[Y]\rangle$.
Its symmetrization is defined by
$$(X,Y)=\langle X,Y\rangle +\langle Y,X\rangle.$$

In this paper we mainly deal with the category $\coh\X$ of coherent sheaves over a weighted
projective line $\X$ of weight type $(p_1, p_2, p_3)$ with $p_i\geq 2$ for $i=1,2,3$.
In this case, we have the following explicit expression for line bundles in $\coh\X$.
\begin{proposition} [\cite{DR25}] \label{Grothendieck group of line bundles in normal form}
Assume that $\vx=\sum\limits_{i=1}^{3}l_{i}\vx_{i}+l\vc$ is in
normal form. Then we have in $K_0(\X)$,
$$[\co(\vx)]=\sum\limits_{i=1}^{3}[\mathcal
{O}(l_{i}\vx_{i})]+l[\co(\vc)] - (l+2)[\co].
$$
\end{proposition}

\subsection{Orthogonal exceptional pairs}

Let $\mathcal{A}$ be an abelian category over a field $\mathbf{k}$. An object $E\in\mathcal{A}$ is
called exceptional if $\Ext^1(E,E)=0$ and $\End(E)\cong \mathbf{k}$. A pair of exceptional
objects $(E_1, E_2)$ is called orthogonal if
$$\Hom(E_1, E_2)=0=\Hom(E_2, E_1)\;\text{ and }\;\, \Ext^1(E_2, E_1)=0.$$
A theorem of Schofield (see \cite{Rin3}) makes it possible to construct exceptional objects
as extensions of orthogonal exceptional pairs. This procedure is often called Schofield
induction which is a very useful tool in the study of representation theory of algebras.

Following \cite{SS1}, an orthogonal exceptional pair $(E_1, E_2)$ is said to be
associated to an exceptional object $E$ if there exists an exact sequence
$$0\lra E_2\lra E\lra E_1\lra 0.$$
Then, in this case,
$$\aligned
 \dim\Ext^1(E_1, &E_2)=\dim\Hom(E_2, E) =\dim\Hom(E, E_1)=1, \\
 & \langle E_2, E_1\rangle=0\;\text{ and }\;\, \langle E_1, E_2\rangle=-1.
\endaligned$$

By a result of Kussin, Lenzing and Meltzer \cite{KLM}, each indecomposable vector
bundle of rank two in $\coh\X$ is an extension bundle, that is, isomorphic to the
middle term of the following ``unique" non-split exact sequence
\begin{align} \label{exbundle}
{0\lra L(\vw)\lra E_L\langle \vx\rangle\lra L(\vx)\lra 0,}
\end{align}
where $L$ is a line bundle, $\vx\in\mathbb{L}$
with $0\leq \vx \leq \sum\limits_{i=1}^{3}(p_i-2)\vx_i$. Moreover, all the extension bundles $E_L\langle
\vx\rangle$ are exceptional in $\coh\X$. Hence, they are determined
by their classes in the Grothendieck group $K_0(\X)$. For simplicity,
we denote $E_{\co}\langle \vx\rangle$ by $E\langle \vx\rangle$ in the rest of this paper.

The following is a key observation on the
feature of extension bundles.

\begin{proposition} [\cite{DR25}] \label{extension bundle which are in the same orbit}
Assume that $\vx, \vy, \vz\in\bbl$ and $\vx=\sum\limits_{i=1}^{3}l_i\vx_i$
with $0\leq l_i\leq p_i-2$ for $1\leq i\leq 3$. Then
$E\langle \vx\rangle=E\langle \vy\rangle(\vz)$ if and only if one of
the following conditions holds:
	\begin{itemize}
		\item[(i)] $\vy=\vx$ and $\vz=0$;
		\item[(ii)] $\vy=l_j\vx_j+\sum\limits_{i\neq
			j}(p_i-2-l_i)\vx_i$ and $\vz=\sum\limits_{i\neq j}(l_i+1)\vx_i-\vc$
		for some $1\leq j\leq 3$.
	\end{itemize}
\end{proposition}

\begin{proposition} \label{orth}
Assume that $L$ is a line bundle and $E=E_L\langle \vx\rangle$ is an extension bundle for
$\vx\in\bbl$. Then $(L(\vx),L(\vw))$ is an orthogonal exceptional pair associated to the
exceptional sheaf $E$, that is,
\[ 	\Hom(L(\vx),L(\vw))= \Hom(L(\vw),L(\vx))=\Ext^1(L(\vw),L(\vx))=0. 	\]
\end{proposition}

\begin{proof} It is known that the dimension of $\Hom(L(\vw),L(\vx))$ is determined
by $\vx-\vw$. Since $\vx$ is an element in $\mathbb{L}$ with
$0\leq \vx \leq \sum\limits_{i=1}^{3}(p_i-2)\vx_i$ and $\vw=\vc-\sum\limits_{i=1}^{3}\vx_i$, we have $$\sum_{i=1}^{3}\vx_i-\vc\leq \vx-\vw \leq \sum\limits_{i=1}^{3}(p_i-1)\vx_i-\vc.$$
Hence, $\Hom(L(\vw),L(\vx))=0$.
	
Similarly, we have
$$\sum\limits_{i=1}^{3}\vx_i-2\vc\leq \vw-\vx \leq \sum\limits_{i=1}^{3}(p_i-1)\vx_i-2\vc.$$
 Thus, $\Hom(L(\vx),L(\vw))=0$.
	
By the Serre duality, we have
$$\Ext^1(L(\vw),L(\vx))\cong D\Hom(L(\vx),L(2\vw)).$$
Since $-\vc\leq 2\vw-\vx \leq \sum\limits_{i=1}^{3}(p_i-2)\vx_i-\vc$, it follows that $\Ext^1(L(\vw),L(\vx))=0.$
\end{proof}

\subsection{Ringel--Hall algebra of a finitary category}

In the rest of this paper, we take the field $\mathbf{k}=\mathbb F_q$, a finite field
of $q$ elements. Let $\mathcal{A}$ be an essentially small abelian category over $\mathbf{k}$.
Assume $\ca$ is finitary, that is,
\[
|\Hom_\ca(M,N)|<\infty,\quad |\Ext^1_\ca(M,N)|<\infty,\,\,\forall\, M,N\in\ca.
\]

Let ${\rm Iso}(\mathcal{A})$ be the set of isomorphic classes in $\mathcal{A}$.
If no confusion arises, we also use $[M]$ to denote the isomorphism class of an object $M$.
The {\em Ringel--Hall algebra} (or {\em Hall algebra} for short) {$\mathcal{H}(\mathcal{A})$} of
$\mathcal{A}$ is defined to be the $\mathbb{Q}$-vector space with the basis
$\big\{[M ]\mid [M]\in{\rm Iso}(\mathcal{A})\big\}$,
endowed with the multiplication defined by (see \cite{Rin1})
$$[A]\cdot [B]=\sum_{[M]\in{\rm Iso}(\mathcal{A})}F_{AB}^M\cdot [M],
\;\;\forall\, A,B\in\mathcal{A},$$
where
$$F_{AB}^M= \big |\{X\subseteq M \mid X \cong B,  M/X\cong A\} \big |.$$
Note that the associativity of multiplication follows from the equalities
\begin{align*}
F_{ABC}^M:=\sum\limits_{[X]}F_{AB}^XF_{XC}^M=\sum\limits_{[Y]}F_{AY}^MF_{BC}^Y,
\end{align*}
where $A,B,C,M\in\ca$.

Furthermore, for given objects $A,B,M\in\ca$, let
$$\Ext_\ca^1(A,B)_M\subseteq \Ext_\ca^1(A,B)$$
denote the subset consisting of extensions whose middle term is isomorphic to $M$. By $a_M$
we denote the cardinality of the automorphism group ${\rm Aut}(M)$ of $M$. The
Riedtmann--Peng's formula states that
 \[F^M_{AB}=\frac{|\Ext_\ca^1(A,B)_M|}{|\Hom_\ca(A,B)|} \cdot \frac{a_M}{a_Aa_B}.
\]

If, moreover, $\ca$ is hereditary, then for each quadruple
$(M,N,X,Y)$ of objects in $\ca$, the following famous \emph{Green's formula}
holds:
\begin{align}\label{Green's formula}
&\sum_{[E]} F_{MN}^{E}F_{XY}^{E}\frac{1}{a_E}
=\sum\limits_{[A],[B],[C],[D]}q^{-\langle A,D\rangle}F_{AB}^{M}F_{CD}^{N}F_{AC}^{X}
F_{BD}^{Y}\frac{a_Aa_Ba_Ca_D}{a_Ma_Na_Xa_Y},
\end{align}
where $\langle A,D\rangle={\rm dim\,Hom}_{\ca}(A,D)-{\rm dim\,Ext}^1_{\ca}(A,D)$ is
the Euler form in $\ca$; see \cite[Thm 2]{Gr}.

\subsection{Two recursion formulas}

The orthogonal exceptional pairs provide an effective way to calculate Hall numbers
in this paper. Keep the notation in 2.3, i.e., $\ca$ is a finitary abelian category
over a finite field $\mathbf{k}=\bbF_q$. We now give two recursion
formulas related to orthogonal exceptional pairs which will be used later on.

\begin{proposition} Let $(E_1,E_2)$ be an orthogonal exceptional
pair in $\mathcal{A}$ associated to an exceptional object $E$. Then for any $X,X'\in\mathcal{A}$,
\[F_{X'E}^X=\sum_{[Z_1]}F_{X'E_1}^{Z_1}F_{Z_1E_2}^X-\sum_{[Z_2]}F_{X'E_2}^{Z_2}F_{Z_2E_1}^X. \]
\end{proposition}

\begin{proof} Since $F_{E_1E_2}^Z=1$ for $Z\cong E$ or $E_1\oplus E_2$, and $F_{E_1E_2}^Z=0$
otherwise, it follows that
    \begin{align*}
F_{X'E_1E_2}^X&=\sum_{[Z]}F_{E_1E_2}^{Z}F_{X'Z}^X\\&=F_{E_1E_2}^{E}F_{X'E}^X+
F_{E_1E_2}^{E_1\oplus E_2}F_{X',E_1\oplus E_2}^X=F_{X'E}^X+F_{X',E_1\oplus E_2}^X.
    \end{align*}
    On the other hand, $F_{X'E_1E_2}^X=\sum_{Z_1}F_{X'E_1}^{Z_1}F_{Z_1E_2}^X$. Thus,
\[ F_{X'E}^X=\sum_{[Z_1]}F_{X'E_1}^{Z_1}F_{Z_1E_2}^X-F_{X',E_1\oplus E_2}^X.    \]
 Similarly, $F_{E_2E_1}^Z=1$ for $Z\cong E_1\oplus E_2$, and $F_{E_2E_1}^Z=0$ otherwise, so
\begin{align*}
F_{X'E_2E_1}^X&=\sum_{[Z]}F_{E_2E_1}^{Z}F_{X'Z}^X\\&=F_{E_2E_1}^{E_1\oplus E_2}F_{X',E_1\oplus E_2}^X
=F_{X',E_1\oplus E_2}^X.
\end{align*}
Moreover, $F_{X'E_2E_1}^X=\sum_{[Z_2]}F_{X'E_2}^{Z_2}F_{Z_2E_1}^X$. This implies that
\[ F_{X'E}^X=\sum_{[Z_1]}F_{X'E_1}^{Z_1}F_{Z_1E_2}^X-\sum_{[Z_2]}F_{X'E_2}^{Z_2}F_{Z_2E_1}^X. \]
\end{proof}

\begin{proposition} \label{simplify}
Let $\mathcal{A}$ be a hereditary abelian category and $(E_1,E_2)$ be an orthogonal
exceptional pair in $\mathcal{A}$ associated to an exceptional object $E$.
Then for any two objects $X,X'\in\mathcal{A}$,
\begin{align*}
F_{X'X}^E=\frac{1}{q-1}\sum_{[X_1],[X_2],[X_3],[X_4]}q^{-\langle X_4,X_1\rangle}&F_{X_3X_1}^{X}F_{X_4X_2}^{X'}\frac{a_{X_1}a_{X_2}a_{X_3}a_{X_4}}{a_{X}a_{X'}}\\ &\times(F_{X_2X_1}^{E_2}F_{X_4X_3}^{E_1}-F_{X_2X_1}^{E_1}F_{X_4X_3}^{E_2}).
\end{align*}
\end{proposition}

\begin{proof} Consider the Green's formula associated to the quadruple $(E_1, E_2,X',X)$:
\begin{align*}
 &\sum_{[W]}F_{X'X}^{W}F_{E_1 E_2}^{W}\frac{1}{a_{W}}\\
 &=\sum_{[X_1],[X_2],[X_3],[X_4]}q^{-\langle X_4,X_1\rangle}F_{X_2X_1}^{E_2}F_{X_3X_1}^{X}F_{X_4X_2}^{X'}
 F_{X_4X_3}^{E_1}\frac{a_{X_1}a_{X_2}a_{X_3}a_{X_4}}{a_{X}a_{X'}a_{E_1}a_{E_2}}.
 \end{align*}
 This can be depicted in the following commutative diagram:
\[ 	\begin{tikzpicture}
	\node (0) at (0,0) {$X_3$};
	\node (1) at (1.5,0) {$E_1$};
	\node (2) at (3,0){$X_4$};
	\node (01) at (0,1) {$X$};
	\node (11) at (1.5,1) {$W$};
	\node (21) at (3,1){$X'$};
	\node (02) at (0,2) {$X_1$};
	\node (12) at (1.5,2) {$E_2$};
	\node (22) at (3,2){$X_2$};	
	\draw[->] (01) --node[above ]{} (11);
	\draw[->] (11) --node[above ]{} (21);
	\draw[->] (12) --node[above ]{} (11);
	\draw[->] (11) --node[above ]{} (1);	
	\draw[->] (0) --node[above ]{} (1);
	\draw[->] (1) --node[above ]{} (2);
	\draw[->] (02) --node[above ]{} (12);
	\draw[->] (12) --node[above ]{} (22);
	\draw[<-] (0) --node[above ]{} (01);
	\draw[<-] (01) --node[above ]{} (02);
	\draw[->] (21) --node[above ]{} (2);
	\draw[->] (22) --node[above ]{} (21);	
	\end{tikzpicture} \]
where all row and column sequences are short exact sequences (For simplicity,
we omit the terms $0$ here).
Since $\Ext^1(E_1,E_2)\cong \mathbf{k}$, we have either $W\cong E$ or
$W\cong E_1\oplus E_2$ when $F_{E_1 E_2}^{W}\not=0$. A direct calculation gives that
$$F_{E_1 E_2}^{E_1\oplus E_2}=F_{E_1E_2}^{E}=1,\, a_{E}=a_{E_1}=a_{E_2}=q-1,\,
a_{E_1\oplus E_2}=(q-1)^2.$$
Therefore,
\begin{align*}
	&\sum_{[X_1],[X_2],[X_3],[X_4]}q^{-\langle X_4,X_1\rangle}F_{X_2X_1}^{E_2} F_{X_3X_1}^{X}
F_{X_4X_2}^{X'}F_{X_4X_3}^{E_1}\frac{a_{X_1}a_{X_2}a_{X_3}a_{X_4}}{a_{X}a_{X'}\cdot(q-1)^2}  \\
 &=F_{X'X}^{E}F_{E_1 E_2}^{E}\frac{1}{q-1}+F_{X'X}^{E_2\oplus E_1}
 F_{E_1 E_2}^{E_2\oplus E_1}\frac{1}{(q-1)^2}  \\
 &=F_{X'X}^{E}\frac{1}{q-1}+F_{X'X}^{E_2\oplus E_1}\frac{1}{(q-1)^2  }.
\end{align*}
Again, by applying Green's formula associated
to the quadruple $(E_2, E_1,X',X)$, we obtain that
\begin{align*}
&\sum_{[W]}F_{X'X}^{W}F_{E_2 E_1}^{W}\frac{1}{a_{W}}\\&=
\sum_{[X_1],[X_2],[X_3],[X_4]}q^{-\langle X_4,X_1\rangle}F_{X_2X_1}^{E_1}F_{X_3X_1}^{X}
F_{X_4X_2}^{X'}F_{X_4X_3}^{E_2}\frac{a_{X_1}a_{X_2}a_{X_3}a_{X_4}}{a_{X}a_{X'}a_{E_2}a_{E_1}}.
	\end{align*}
Since $\Ext^1(E_2,E_1)=0$, the left-hand side contains only one possible
nonzero term, that is, $W\cong E_1\oplus E_2$, and thus,
\begin{align*}
&\sum_{[W]}F_{X'X}^{E_2\oplus E_1}\frac{1}{(q-1)^2} \\
&=\sum_{[X_1],[X_2],[X_3],[X_4]} q^{-\langle X_4,X_1\rangle}F_{X_2X_1}^{E_1}F_{X_3X_1}^{X}
F_{X_4X_2}^{X'}F_{X_4X_3}^{E_2}\frac{a_{X_1}a_{X_2}a_{X_3}a_{X_4}}{a_{X}a_{X'}\cdot (q-1)^2}.
	\end{align*}
By combining two equalities above together, we conclude that
 \begin{align*}
     F_{X'X}^E=\frac{1}{q-1}\sum_{[X_1],X_2,X_3,X_4}q^{-\langle X_4,X_1\rangle}&F_{X_3X_1}^{X}F_{X_4X_2}^{X'}\frac{a_{X_1}a_{X_2}a_{X_3}a_{X_4}}{a_{X}a_{X'}}\\ &\times(F_{X_2X_1}^{E_2}F_{X_4X_3}^{E_1}-F_{X_2X_1}^{E_1}F_{X_4X_3}^{E_2}).
 \end{align*}
\end{proof}

\section{Hall polynomials for line bundles and torsion sheaves}

This section is devoted to calculating Hall numbers for line bundles and torsion
sheaves on a weighted projective line over a finite field $\mathbf{k}$ by applying
Green's formula. It turns out that these numbers are indeed polynomials
in the cardinality of the ground field $\mathbf{k}$ and thus, called Hall polynomials.

Throughout this section, $\X$ always denotes a weight projective line of weight
type ${\bf p}=(p_{1},p_{2},\ldots, p_{t})$ over a field $\mathbf k$.
We keep all the notation in the previous section.

\subsection{Hom and Ext spaces of line bundles and torsion sheaves}
 In this subsection we calculate the dimensions of morphism spaces from
line bundles to torsion sheaves. To this aim, we need some notations. Given a
rational number $m/n$, define
$$\lfloor\dfrac{m}{n}\rfloor={\rm max}\{a\in\Z\mid a\leq m/n\},\;\;
\lceil\dfrac{m}{n}\rceil={\rm min}\{a\in\Z\mid a\geq m/n\}.$$
Further, fix a positive integer $p$, and for $m\in\Z$, set $0\leq \wtd{m}\leq p-1$
such that $\wtd{m}\equiv m\,({\rm mod}\,p)$. Then, for $m,n\in\Z$, we have
the following equalities:
\begin{equation} \label{equality-upper-lower-integer}
\aligned
& \lceil -\dfrac{n}{p}\rceil=-\lfloor\dfrac{n}{p}\rfloor, \quad
\lceil \dfrac{n+1}{p}\rceil=\lfloor\dfrac{n}{p}\rfloor+1,\\
& \lceil\dfrac{n-\wtd{m}}{p}\rceil=\lceil\dfrac{n-{m}}{p}\rceil+\lfloor\dfrac{m}{p}\rfloor,\\
& \lceil\dfrac{n+\wtd{m}}{p}\rceil=\lceil\dfrac{n+{m}}{p}\rceil-\lfloor\dfrac{m}{p}\rfloor.
\endaligned
\end{equation}

Recall from 2.1 the exceptional tubes $\cT_i$ ($i\in I$) in $\coh\X$. In what follows,
for $i\in I, j\in\bbZ$,  and $n\geq 1$, $S_{i,j}^{(n)}$ should be understood as
$S_{i,\tilde j}^{(n)}$ with $p=p_i$.

\begin{lemma} \label{EulerForm} Let $\mathcal{T}_i$ be an exceptional tube in $\coh\X$ of
rank $p_i=p$, and $S_{i,j}^{(n)}, S_{i,k}^{(m)}$ be indecomposable torsion sheaves
for some $n,m\ge1$ and $j,k\in\mathbb{Z}$.
\begin{itemize}
\item[(1)] If $n\ge m\ge1$, then
$$\dim\Hom(S_{i,j}^{(n)}, S_{i,k}^{(m)})=\lceil\dfrac{m-{k+j}}{p}\rceil+\lfloor\dfrac{k-j}{p}\rfloor,$$
\[
\dim\Ext^1(S_{i,k}^{(m)}, S_{i,j}^{(n)})=\lfloor\dfrac{m-{k+j}}{p}\rfloor+\lceil\dfrac{k-j}{p}\rceil.
\]

\item[(2)] If $m\ge n\ge1$, then
$$\dim\Hom(S_{i,j}^{(n)}, S_{i,k}^{(m)})=\lceil\dfrac{m-{k+j}}{p}\rceil+\lfloor\dfrac{n-m+k-j}{p}\rfloor,$$
\[
\dim\Ext^1(S_{i,k}^{(m)}, S_{i,j}^{(n)})=\lfloor\dfrac{m-{k+j}}{p}\rfloor+\lceil\dfrac{n-m+k-j}{p}\rceil.
\]

\item[(3)] For any $n,m\ge1$,
$$\langle S_{i,j}^{(n)}, S_{i,k}^{(m)}\rangle=\lceil\dfrac{m-{k+j}}{p}\rceil
+\lfloor\dfrac{k-j}{p}\rfloor-\lfloor\dfrac{n-{j+k}}{p}\rfloor-\lceil\dfrac{m-n+j-k}{p}\rceil.$$
\end{itemize}
\end{lemma}

\begin{proof} (1) Since each morphism $S_{i,j}^{(n)}\to S_{i,k}^{(m)}$ factors
through the canonical embedding $S_{i,j}^{(m-(k-j)^\sim)}\to S_{i,k}^{(m)}$,
we obtain the equalities
$$\dim\Hom(S_{i,j}^{(n)}, S_{i,k}^{(m)})=\dim\Hom(S_{i,j}^{(n)},
S_{i,j}^{(m-(k-j)^\sim)})=\lceil\dfrac{m-(k-j)^\sim}{p}\rceil.$$
By \eqref{equality-upper-lower-integer}, we obtain that
 \[  \lceil\dfrac{m-(k-j)^\sim}{p}\rceil=\lceil\dfrac{m-{k+j}}{p}\rceil+\lfloor\dfrac{k-j}{p}\rfloor.   \]
Furthermore, by Serre duality,
$$\dim\Ext^1(S_{i,k}^{(m)}, S_{i,j}^{(n)})=\dim\Hom(S_{i,j}^{(n)}, S_{i,k-1}^{(m)})=\lceil\dfrac{m-{k+1+j}}{p}\rceil+\lfloor\dfrac{k-1-j}{p}\rfloor.$$
Again by \eqref{equality-upper-lower-integer}, this equals to
$$\lfloor\dfrac{m-{k+j}}{p}\rfloor+1+\lceil\dfrac{k-j}{p}\rceil-1
=\lfloor\dfrac{m-{k+j}}{p}\rfloor+\lceil\dfrac{k-j}{p}\rceil.$$

(2) In this case, each morphism $S_{i,j}^{(n)}\to S_{i,k}^{(m)}$ factors through the canonical
embedding $S_{i,j}^{(n-(n-m+k-j)^\sim)}\to S_{i,k}^{(m)}$. Thus, we obtain the equalities
$$\dim\Hom(S_{i,j}^{(n)}, S_{i,k}^{(m)})=\dim\Hom(S_{i,j}^{(n)}, S_{i,j}^{(n-(n-m+k-j)^\sim)})=\lceil\dfrac{n-(n-m+k-j)^\sim}{p}\rceil.$$
 This implies by \eqref{equality-upper-lower-integer} that
\[\lceil\dfrac{n-(n-m+k-j)^\sim}{p}\rceil=\lceil\dfrac{m-{k+j}}{p}\rceil+\lfloor\dfrac{n-m+k-j}{p}\rfloor.   \]
Applying Serre duality gives that
$$\aligned
\dim\Ext^1(S_{i,k}^{(m)}, S_{i,j}^{(n)})&=\dim\Hom(S_{i,j}^{(n)}, S_{i,k-1}^{(m)})\\
&=\lceil\dfrac{m-{k+1+j}}{p}\rceil+\lfloor\dfrac{n-m+k-1-j}{p}\rfloor,
\endaligned$$
which equals to $\lfloor\dfrac{m-{k+j}}{p}\rfloor+\lceil\dfrac{n-m+k-j}{p}\rceil$ by
\eqref{equality-upper-lower-integer}.

(3) It follows immediately from (1), (2) and the equality
$$\langle S_{i,j}^{(n)}, S_{i,k}^{(m)}\rangle=\dim\Hom(S_{i,j}^{(n)}, S_{i,k}^{(m)})
-\Ext^1 (S_{i,j}^{(n)}, S_{i,k}^{(m)}).$$
\end{proof}

\begin{lemma} \label{lem3.2}
Let $L$ be a line bundle and $S, S'$ be indecomposable torsion sheaves supported
at the same point. Assume $S$ and $S'$ have length $n$ and $n'$, respectively,
with $n\leq n'$. If $\Hom(L,\mathrm{top\,} S')\ne0$, then
$$\Hom(L,S)\cong \Hom(S',S)\; \text{ and }\;\, \Hom(L,S')\cong \End(S').$$
\end{lemma}

\begin{proof} If $S=S_z^{(n)}$ is an indecomposable torsion sheaf
of class $dn\delta$ lying in a homogeneous tube $\mathcal{T}_z$, where $z\in\bbH_\bfk$ with $\mathrm{deg}\,z=d$,
then we obtain exact sequences
\[0\lra \Hom(L,S_z^{(1)})\lra \Hom(L,S_z^{(t)})\lra \Hom(L,S_z^{(t-1)})\lra 0   \]
 for $2\leq t\leq n$. We conclude that
$$\Hom(L,S)\cong \bigoplus_{k=1}^n\Hom(L,S_z^{(1)})\cong \mathbf{k}^{dn}\cong \Hom(S',S).$$

If $S=S_{i,{j}}^{(n)}$ is an indecomposable torsion sheaf lying in an
exceptional tube $\mathcal{T}_i$ of rank $p_i$ for some $i\in I$, then there are exact sequences
\[ 0\lra \Hom(L,S_{i,{j-n+t-1}}^{(t-1)})\lra \Hom(L,S_{i,{j-n+t}}^{(t)})
\lra \Hom(L,S_{i,{j-n+t}}^{(1)})\lra 0    \]
 for $2\leq t\leq n$. Thus,
$$\Hom(L,S)\cong \Hom(L,S_{i,{j}}^{(1)})\oplus \Hom(L,S_{i,{j-1}}^{(1)})
\oplus \cdots \oplus \Hom(L,S_{i,j-n+1}^{(1)}).$$
Assume $\mathrm{top}\,S'=S_{i,{j'}}^{(1)}$ for some $0\leq j'\leq p_i-1$.
Then $S'=S_{i,{j'}}^{(n')}$, and by Lemma \ref{EulerForm},
$$\Hom(L,S)\cong \mathbf{k}^m\cong \Hom(S',S),$$
 where $m=\lceil\frac{n-(j-j')^\sim}{p_i}\rceil$.

This finishes the proof of the first isomorphism, while the
second one follows immediately by replacing $S$ by $S'$.
\end{proof}

Consequently, we obtain the following result,  which indicates that we can use the formulas in Lemma \ref{EulerForm} to calculate the homomorphisms (resp. extensions) between line bundles and torsion sheaves. The strategy will be used frequently later on.

\begin{corollary} \label{cor3.3}
Let $L$ be a line bundle such that $\Hom(L, S_{i,j})\neq 0$ for some simple exceptional
sheaf $S_{i,j}$. Then for $n\geq 1$,
$$\Hom(L,S_{i,k}^{(n)})\cong \Hom(S_{i,j}^{(n)},S_{i,k}^{(n)})
\text{\;\; and \;\;} \Ext^1(S_{i,k}^{(n)},L)\cong \Ext^1(S_{i,k}^{(n)},S_{i,j}^{(n)}).$$
\end{corollary}

\begin{proof}
    The first statement follows from Lemma \ref{lem3.2}, while  the second one follows from Serre duality.
\end{proof}

\subsection{Hall numbers for line bundles and torsion sheaves} We first consider
some easy cases where only line bundles and torsion sheaves are involved. In this subsection,
$\X$ is an arbitrary weight projective line over a finite field $\mathbf{k}=\bbF_q$.

\begin{proposition} \label{hallpoly1}
Assume that $L$ is a line bundle, $\vx, \vy \in\bbl$ and $S=\bigoplus_k S_k$ is a torsion
sheaf, where $S_k$ are all indecomposable. Then there exists an exact sequence
 \begin{align}\label{exact L xy}
     0\lra L(\vx)\lra L(\vy)\lra S\lra 0
 \end{align}
	if and only if the following three conditions hold:
 \begin{itemize}
     \item [(1)]  $[L(\vx)]+[S]=[ L(\vy)]$ in $\mathrm{K}_0(\X)$;
      \item [(2)] the $S_k$ lie in pairwise distinct tubes;
      \item [(3)] $\Hom(L(\vy),\mathrm{top\,} S_k)\ne0$ for each $k$.
 \end{itemize}
Moreover, in this case,
\[F_{SL(\vx)}^{L(\vy)}=1.	 \]
\end{proposition}

\begin{proof}
First, the exact sequence \eqref{exact L xy} gives that $[L(\vx)]+[S]=[ L(\vy)]$.
Further, the existence of the epimorphism $L(\vy)\to S$ implies that $S_k$ must be taken from pairwise
distinct tubes, c.f. \cite[Exam. 4.12]{Sch2}. Thus, there is an epimorphism from $L(\vy)$
to $S_k$ for each $k$. We conclude that $\Hom(L(\vy),\mathrm{top\,} S_k)\ne0$.

Conversely, $\Hom(L(\vy),\mathrm{top\,} S_k)\ne0$ implies that there exists an epimorphism
$\pi_k:L(\vy)\rightarrow S_k$. By \cite[Exam. 4.12]{Sch2}, there is an epimorphism
$$(\pi_k):L(\vy)\lra S=\bigoplus_k S_k,$$
 and its kernel $K$ is a line bundle. Thus, $[K]+[S]=[ L(\vy)]$. This together with
$[L(\vx)]+[S]=[ L(\vy)]$ implies that $K\cong L(\vx)$ and hence, we obtain an exact sequence
of form \eqref{exact L xy}. Consequently, $F_{SL(\vx)}^{L(\vy)}=1$.
\end{proof}

Note that in the proposition above, if $[S_k]\ne c\delta$ for each $c\ge0$, then
the condition that $\Hom(L(\vy),\mathrm{top\,} S_k)\ne0$ follows from the other two conditions.

\medskip

Next, we consider the case where the middle term is given by a direct sum of a line bundle and
a torsion sheaf. It turns out that if such an exact sequence does not split, then the
associated Hall number is related to the cardinalities of automorphism groups.

\begin{proposition} \label{hallpoly1'}
Assume that $L,L'$ are line bundles and $S$ is a torsion sheaf. Then there exist a
torsion sheaf $S'$ (may be zero) and an exact sequence
\[0\lra L\lra L'\oplus S'\lra S\lra 0 	\]
if and only if $S'$ is a subsheaf of $S$ such that there is an exact sequence
\[0\lra L\lra L'\lra S/S'\lra 0. 	\]
 In this case, by setting $S''=S/S'$, we have
 \[
 F_{SL}^{L'\oplus S'}={|\Hom(L',S')|}\cdot F_{S''S'}^{S}\frac{ a_{S'}a_{S''}}{ a_S }.
 \]
 In particular, if $S$ is indecomposable, then
 \[
 F_{SL}^{L'\oplus S'}=\begin{cases}
        |\Hom(L,S)|,& \text{ if } S\cong S';\\
        a_{S'},& \text{ if } S\ncong S'.
     \end{cases}
    \]
\end{proposition}

\begin{proof} By applying Green's formula to the quadruple $(L',S',S,L)$, we obtain that
\begin{align}\label{Green formula LLSS}
	&\sum_{[W]}F_{L'S'}^{W}F_{SL}^{W}\frac{1}{a_{W}}\\\notag
 &=\sum_{[X_1],[X_2],[X_3],[X_4]}q^{-\langle X_4,X_1\rangle}F_{X_2X_1}^{S'}F_{X_3X_1}^{L}F_{X_4X_2}^{S}F_{X_4X_3}^{L'}
 \frac{a_{X_1}a_{X_2}a_{X_3}a_{X_4}}{a_{L}a_{L'}a_{S}a_{S'}};
\end{align}
see the following commutative diagram:
%Let $S'$ be a subsheaf of $S$ with the quotient sheaf $S''=S/S'$. Then, by Green's formula, we have
\begin{align}\label{commutative diag}
	\begin{tikzpicture}
	\node (0) at (0,0) {$X_3$};
	\node (1) at (1.5,0) {$L'$};
	\node (2) at (3,0){$X_4$};
	\node (01) at (0,1) {$L$};
	\node (11) at (1.5,1) {$W$};
	\node (21) at (3,1){$S$};
	\node (02) at (0,2) {$X_1$};
	\node (12) at (1.5,2) {$S'$};
	\node (22) at (3,2){$X_2$};	
	\draw[->] (01) --node[above ]{} (11);
	\draw[->] (11) --node[above ]{} (21);
	\draw[->] (12) --node[above ]{} (11);
	\draw[->] (11) --node[above ]{} (1);	
	\draw[->] (0) --node[above ]{} (1);
	\draw[->] (1) --node[above ]{} (2);
	\draw[->] (02) --node[above ]{} (12);
	\draw[->] (12) --node[above ]{} (22);
	\draw[<-] (0) --node[above ]{} (01);
	\draw[<-] (01) --node[above ]{} (02);
	\draw[->] (21) --node[above ]{} (2);
	\draw[->] (22) --node[above ]{} (21);	
	\end{tikzpicture}
\end{align}
 where all row and column sequences are short exact sequences.
Then $W\cong L'\oplus S'$ and $X_1=0$.
Thus, $X_2\cong S'$ and $ X_3\cong L$. Since $X_4$ is a torsion sheaf and there is
an epimorphism from the line bundle $L'$ to $X_4$, the indecomposable direct summands
of $X_4$ lie in distinct tubes, and they are uniquely determined by $S$ and $S'$. In other words,
$X_4$ is unique (up to isomorphism) and so, $X_4\cong S''$.
Further, $F_{L'S'}^{L'\oplus S'}=1$ since $\Hom(S',L')=0$. In conclusion, the equality in
\eqref{Green formula LLSS} is reduced to the equality
 \begin{align*}
F_{SL}^{L'\oplus S'}&=a_{L'\oplus S'}\cdot F_{S''L}^{L'}F_{S''S'}^{S}\frac{a_{L}a_{S'}a_{S''}}{a_{L}a_{L'}a_{S}a_{S'}}.
	\end{align*}
Hence, $F_{SL}^{L'\oplus S'}\neq 0$ if and only if $ F_{S''L}^{L'}F_{S''S'}^{S}\neq 0$.
This proves the first assertion.

 In this case, by Proposition \ref{hallpoly1} we have
 $F_{S''L}^{L'}=1$. It follows that
\begin{align}\label{Hall number for L oplus S}
	F_{SL}^{L'\oplus S'}={|\Hom(L',S')|}\cdot F_{S''S'}^{S}\frac{ a_{S'}a_{S''}}{ a_S }.
	\end{align}

Now assume that $S$ is indecomposable, then so are $S'$ and $S''$. This implies that
$F_{S''S'}^{S}=1$. If $S'\cong S$, then $S''=0$ and $$F_{SL}^{L'\oplus S'}=|\Hom(L',S')|=|\Hom(L,S)|.$$
 If $S'\ncong S$, then $\mathrm{top\,}S=\mathrm{top\,}S''$, and
so $\Hom(L',\mathrm{top\,}S'')=\Hom(L',\mathrm{top\,}S)\ne0$. By Lemma \ref{lem3.2}, we obtain that
$$\Hom(L',S)\cong \End(S)\;\,\text{ and }\;\, \Hom(L',S'')\cong \End(S'').$$
 The exact sequence
 \[  0\lra \Hom(L',S')\lra \Hom(L',S)\lra \Hom(L',S'')\lra 0  \]
together with \eqref{Hall number for L oplus S} gives that
\[ F_{SL}^{L'\oplus S'}=\frac{|\End( S)|}{|\End( S'')|}\cdot \frac{ a_{S'}a_{S''}}{ a_S }. \]
 Since both $S$ and $S''$ are indecomposable and uniserial, we have that
$$|\End( S)|=q^{s+1},\; |\End( S'')|=q^{t+1},\;a_{S}=(q-1)q^s,\;\text{ and }\;\,
a_{S}=(q-1)q^t,$$
 where $s=\dim\rad(S)$ and $t=\dim\rad(S'')$. Therefore,
\[ F_{SL}^{L'\oplus S'}=\frac{|\End( S)|}{|\End( S'')|}\cdot \frac{ a_{S'}a_{S''}}{ a_S }
= \dfrac{ a_{S}}{ a_{S''} } \cdot \frac{ a_{S'}a_{S''}}{ a_S } =a_{S'}. \]
\end{proof}

\begin{proposition} \label{hallpoly1''}
Assume that $L,L'$ are line bundles and $S,S',S''$ are torsion sheaves (or zero sheaves).
Then there exists an exact sequence
\[	0\lra L\oplus S\lra L'\oplus S'\lra S''\lra 0 \]
if and only if there are sheaves $S_1$ and $S_2$ with exact sequences
\[
 0\to S\to S'\to S_1\to 0,\quad 0\to L\to L'\to S_2\to 0,\quad    0\to S_1\to S''\to S_2\to 0.
    \]
   In this case, we have
    \[
    F_{S''L\oplus S}^{L'\oplus S'}=\frac{|\Hom(L',S')|}{|\Hom(L',S)|}\cdot \sum_{X_2}F_{X_2 S}^{S'}F_{S_2 X_2}^{S''}\frac{a_{X_2}a_{S_2}}{a_{S''}}.\]
    In particular, if both $S'$ and $S''$ are indecomposable, then
    \[
 F_{S''L\oplus S}^{L'\oplus S'}=\begin{cases}
        |\Hom(L,S'')|,& \text{ if } L\cong L';\\
        a_{S'/S},& \text{ if } L\ncong L'.
     \end{cases}
    \]
\end{proposition}

\begin{proof} First of all, Green's formula associated to the quadruple $(L',S',S'',L\oplus S)$
has the form
\begin{align}\label{Green formula for LSW}
 &\sum_{[W]}F_{L'S'}^{W}F_{S''L\oplus S}^{W}\frac{1}{a_{W}}\\\notag
 &=\sum_{[X_1],[X_2],[X_3],[X_4]}q^{-\langle X_4,X_1\rangle}F_{X_2X_1}^{S'}F_{X_3X_1}^{L\oplus S}F_{X_4X_2}^{S''}F_{X_4X_3}^{L'}\frac{a_{X_1}a_{X_2}a_{X_3}a_{X_4}}{a_{L\oplus S}a_{L'}a_{S''}a_{S'}}
\end{align}
which is depicted by the following commutative diagram
  \begin{align}\label{comm diag for LSSLS}
	\begin{tikzpicture}
	\node (0) at (0,0) {$X_3$};
	\node (1) at (1.5,0) {$L'$};
	\node (2) at (3,0){$X_4$};
	\node (01) at (0,1) {$L\oplus S$};
	\node (11) at (1.5,1) {$W$};
	\node (21) at (3,1){$S''$};
	\node (02) at (0,2) {$X_1$};
	\node (12) at (1.5,2) {$S'$};
	\node (22) at (3,2){$X_2$};	
	\draw[->] (01) --node[above ]{} (11);
	\draw[->] (11) --node[above ]{} (21);
	\draw[->] (12) --node[above ]{} (11);
	\draw[->] (11) --node[above ]{} (1);	
	\draw[->] (0) --node[above ]{} (1);
	\draw[->] (1) --node[above ]{} (2);
	\draw[->] (02) --node[above ]{} (12);
	\draw[->] (12) --node[above ]{} (22);
	\draw[<-] (0) --node[above ]{} (01);
	\draw[<-] (01) --node[above ]{} (02);
	\draw[->] (21) --node[above ]{} (2);
	\draw[->] (22) --node[above ]{} (21);	
	\end{tikzpicture}
	\end{align}
 Since $X_1$ is a torsion sheaf and $X_3$ is a subsheaf of $L'$, this implies that
$X_3\cong L$ and $X_1\cong S$. As a quotient sheaf of $L'$, the indecomposable
direct summands of $X_4$ lie in distinct tubes, and they are uniquely determined
by $S, S'$ and $S''$. In other words, $X_4$ is a torsion sheaf, written as $S_2$,
and there is an exact sequence
 \begin{align}\label{short exact L L' S2}
\xymatrix{0\lra L\longrightarrow L'\longrightarrow X_4=S_2\lra 0}
 \end{align}
Since $F_{LS}^{L\oplus S}=1=F_{L'S'}^{L'\oplus S'}$, the formula \eqref{Green formula for LSW}
is reduced to the equality
% \textcolor{blue}{Through an argument similar to that in Proposition \ref{hallpoly1'}, we find that $X_4$ must be isomorphic to $S_2$. However, $X_2$ is not unique up to isomorphism.} Hence
 \begin{align}\label{Hall number F L'S'}
	F_{S''L\oplus S}^{L'\oplus S'}&=q^{-\langle S_2,S\rangle} \sum_{[X_2]}F_{X_2 S}^{S'}F_{S_2 X_2}^{S''}\frac{a_{S}a_{X_2}a_{S_2}a_L\cdot a_{L'\oplus S'}}{a_{S'}a_{L\oplus S}a_{S''}a_{L'}}.
	\end{align}
Therefore,  $F_{S''L\oplus S}^{L'\oplus S'}\neq 0$ if and only if there exists a torsion
sheaf $S_1$ such that $F_{S_1 S}^{S'}F_{S_2 S_1}^{S''}\neq 0$. This proves the first statement.

Applying $\Hom(-,S)$ to the sequence in \eqref{short exact L L' S2} yields an exact sequence
\[  0\lra \Hom(S_2,S)\lra \Hom(L',S)\lra \Hom(L,S)\lra \Ext^1(S_2,S)\lra 0.  \]
 Then \eqref{Hall number F L'S'} can be written as
 \begin{align*}
	F_{S''L\oplus S}^{L'\oplus S'}
 &=\frac{|\Ext^1(S_2,S)|}{|\Hom(S_2,S)|}\frac{|\Hom(L',S')|}{|\Hom(L,S)|}\cdot\sum_{[X_2]}F_{X_2 S}^{S'}F_{S_2 X_2}^{S''}\frac{a_{X_2}a_{S_2}}{a_{S''}}
 \\&=\frac{|\Hom(L',S')|}{|\Hom(L',S)|}\cdot \sum_{[X_2]}F_{X_2 S}^{S'}F_{S_2 X_2}^{S''}
 \frac{a_{X_2}a_{S_2}}{a_{S''}}.
	\end{align*}

If, in addition, both $S'$ and $S''$ are indecomposable, then $S$ is indecomposable, and
$X_2$ is a uniquely determined indecomposable torsion sheaf, written as $S_1$. Moreover,
$$F_{S_1 S}^{S'}=1=F_{S_2 S_1}^{S''}.$$
If $L\cong L'$, then $S_2=0$, $S''\cong S_1$, and
$$F_{S''L\oplus S}^{L'\oplus S'}=|\Hom(L',S_1)|=|\Hom(L,S'')|.$$
 If $L\ncong L'$, i.e., $S_2\ne0$, then $\Hom(L',\mathrm{top\,}S_2)=\Hom(L',\mathrm{top\,}S'')\ne0$.
Since $S''$ and $S_2$ are indecomposable and uniserial, we have by Lemma \ref{lem3.2} that
$$\dfrac{|\Hom(L',S')|}{|\Hom(L',S)|}=\dfrac{|\Hom(L',S'')|}{|\Hom(L',S_2)|}
 =\dfrac{|\End(S'')|}{|\End(S_2)|}=\dfrac{a_{S''}}{a_{S_2}}.$$
 Consequently, we have
 \begin{align*}
	F_{S''L\oplus S}^{L'\oplus S'}&=\frac{a_{S''}}{a_{S_2}}\cdot F_{S_1 S}^{S'}F_{S_2 S_1}^{S''}\frac{a_{S_1}a_{S_2}}{a_{S''}}=a_{S_1}=a_{S'/S}.
	\end{align*}
\end{proof}

\section{Hall polynomials involving extension bundles}

In this section we assume that $\X$ is a weighted projective line of weight type
$(p_1,p_2,p_3)$ over a finite field $\mathbf{k}=\bbF_q$, and focus on calculating
Hall numbers for the case where extension bundles are concerned with. As in the
previous section, these numbers are again polynomials in $q=|\bfk|$.

By \eqref{exbundle} in Section 2, each bundle of rank 2 is an extension of two line bundles
$L(\vw)$ and $L(\vx)$ for $0\leq \vx \leq \sum\limits_{i=1}^{3}(p_i-2)\vx_i$.
We first describe the morphism and extension spaces between $L(\vw)$ and $L(\vx)$.

By the orthogonality of $(L(\vx),L(\vw))$, we immediately obtain that
$F_{L(\vx)L(\vw)}^{W}=1$ if $W\cong E$ or $W\cong L(\vx)\oplus L(\vw)$, and
$F_{L(\vx)L(\vw)}^{W}=0$ for the remaining cases. Moreover, $F_{L(\vw)L(\vx)}^{W}=1$ if
$W\cong L(\vx)\oplus L(\vw)$, and $F_{L(\vw)L(\vx)}^{W}=0$ otherwise.
\medskip

\begin{lemma} Let $L_1$ and $L_2$ be two line bundles in $\coh\X$. Then
$$\Ext^1(L_1,L_2)=0\;\,\text{ or }\;\,\Ext^1(L_2,L_1)=0.$$
\end{lemma}

\begin{proof} Suppose that $\Ext^1(L_1,L_2)\ne0$ and $\Ext^1(L_2,L_1)\ne0$. Then
\[     \det L_1- \det L_2+\vw\ge0,\quad \det L_2- \det L_1+\vw\ge0.     \]
 This implies that $ 2\vw=2\vc-2\sum_{i=1}^3\vx_i\geq 0$, a contradiction.
\end{proof}

By applying Proposition \ref{simplify} to the orthogonal exceptional pair $(L(\vx),L(\vw))$
associated to $E=E_L\langle\vx\rangle$ in $\coh\X$, we obtain that for any $X,X'\in\coh\X$,
 \begin{equation} \label{Hall-number-ext-bdl-1}
 \aligned
 F_{X'X}^E=\frac{1}{q-1}\sum_{[X_1],[X_2],[X_3],[X_4]}q^{-\langle X_4,X_1\rangle}
   &F_{X_3X_1}^{X}F_{X_4X_2}^{X'}\frac{a_{X_1}a_{X_2}a_{X_3}a_{X_4}}{a_{X}a_{X'}}\\ &\times(F_{X_2X_1}^{L(\vw)}F_{X_4X_3}^{L(\vx)}-F_{X_2X_1}^{L(\vx)}F_{X_4X_3}^{L(\vw)}).
\endaligned
\end{equation}
In order to calculate $F_{X'X}^E$, we write
\begin{equation} \label{Hall-number-ext-bdl-2}
F_{X'X}^E=\Psi_1-\Psi_2,
\end{equation}
 where
$$\aligned
\Psi_1&=\frac{1}{q-1}\sum_{[X_1],[X_2],[X_3],[X_4]}q^{-\langle X_4,X_1\rangle}F_{X_3X_1}^{X}F_{X_4X_2}^{X'}
\frac{a_{X_1}a_{X_2}a_{X_3}a_{X_4}}{a_{X}a_{X'}}F_{X_2X_1}^{L(\vw)}F_{X_4X_3}^{L(\vx)},\\
\Psi_2&=\frac{1}{q-1}\sum_{[X_1],[X_2],[X_3],[X_4]}q^{-\langle X_4,X_1\rangle}F_{X_3X_1}^{X}F_{X_4X_2}^{X'}
\frac{a_{X_1}a_{X_2}a_{X_3}a_{X_4}}{a_{X}a_{X'}}F_{X_2X_1}^{L(\vx)}F_{X_4X_3}^{L(\vw)}.
\endaligned$$

\medskip

The automorphism groups of torsion sheaves and their cardinalities play a
role in the calculation of Hall numbers. We first present some formulas given in
\cite[Prop.~6.1]{SS2}. As pointed out in Section 2, there are exactly three exceptional tubes
$\cT_1,\cT_2,\cT_3$ of ranks $p_1,p_2, p_3$, respectively, in $\coh\X$. Then the sums
of classes of all simple objects in $\cT_1$,
$\cT_2$ and $\cT_3$ in $\mathrm{K}_0(\X)$ are the same, equal to $\delta$. For $1\leq i\leq 3$, take an
indecomposable torsion sheaf $S_i$ in $\cT_i$ with the class $\sigma_i<\delta$.

For $0\leq k\leq 3$, let $\mathscr{S}_{n}^{\sigma_1,\ldots,\sigma_k}$ denote
the set of isomorphism classes of all torsion sheaves $S$ such that
\begin{itemize}
\item[(1)] the class of $S$ is $n\delta+\sigma_1+\cdots+\sigma_k$,
\item[(2)] its indecomposable summands are taken from pairwise distinct tubes, and
\item[(3)] $S$ contains exactly one indecomposable summand from each
exceptional tube of $\cT_1, \ldots, \cT_k$.
\end{itemize}
Finally, set
$$s_n^{(k)}(q)=\frac{1}{q-1}\sum_{[S]\in \mathscr{S}_{n}^{\sigma_1,\ldots,\sigma_k}}a_S.$$
 By convention, set $s_0^{(k)}(q)={(q-1)}^{k-1}$ and $s_{-1}^{(k)}(q)=0$.

Further, define
\begin{equation} \label{def-f(q)}
f_n(q)=\sum_{i=1}^n (-1)^{i-1}(2i-1)q^{n+1-i}+(-1)^n (n+1),\;\;\forall\, n\geq 1,
\end{equation}
and set $f_0(q)=1,\,f_{-n}(q)=0$.

\begin{proposition} [{\rm \cite[Prop.~6.1]{SS2}}] \label{ssprop}
For $n\ge1$ and $0\leq k\leq 3$, we have the following equalities
\[ s_n^{(k)}(q)=\sum_{i=1}^{2n+k-1} (-1)^{i-1}i q^{2n+k-i}-(-1)^k(n+1).  \]
Moreover,
 \begin{align*}
    & s_n^{(0)}(q)-s_{n-2}^{(3)}(q)=f_{2n-1}(q), & \text{ for }n\ge1;
	\\&s_{n-1}^{(3)}(q)-s_{n}^{(0)}(q)=f_{2n}(q), & \text{ for }n\ge1;
	\\&s_{n}^{(1)}(q)-s_{n-1}^{(2)}(q)=f_{2n}(q), & \text{ for }n\ge0;
	\\&s_{n}^{(2)}(q)-s_{n}^{(1)}(q)=f_{2n+1}(q), & \text{ for }n\ge0,
	\end{align*}
and
 \begin{align}
 \label{sumf}
     (q-1)\sum_{t\ge0}q^tf_{n-2t}(q)=f_{n+1}(q)+(-1)^n \; \text{ for }n\ge0.
 \end{align}

\end{proposition}

In the following we view each $f_n(q)$ either as a polynomials in $q$ or
as a number evaluating at the cardinality $q=|\bfk|$ of a finite field $\bfk$.
For notational simplicity, we sometimes write $f_n$ instead of $f_n(q)$.

\medskip

With the help of notations above, we are ready to calculate Hall numbers (polynomials)
involving extension bundles.

\subsection{Hall polynomial $F_{L_2L_1}^E$ for line bundles $L_1$ and $L_2$}

In this subsection, we calculate Hall number $F_{L_2L_1}^E$, where $E$ is
an extension bundle and $L_1$ and $L_2$ are line bundles.
It turns out that it is determined by the value $\langle L_1,L_2\rangle$.

\begin{theorem} \label{hallpoly2}
Assume that $L_1,L_2$, and $L$ are line bundles and $E=E_L\langle \vx\rangle$ is
an extension bundle with $\vx\in\bbl$. If $[L_1]+[L_2]=[E]$, then
$F_{L_2L_1}^E=f_{\langle L_1,L_2\rangle}(q).$
\end{theorem}

\begin{proof}
We first consider the case $L_1\cong L(\vx)$ or $L(\vw)$. If $L_1\cong L(\vx)$,
then $L_2\cong L(\vw)$. In this case, by Proposition \ref{orth}, we have
$$\langle L_1,E\rangle=\langle [L_1],[L_1]+[L_2]\rangle=0\;\text{ and }\;\,
F_{L_2L_1}^E=0.$$
If $L_1\cong L(\vw)$, then $L_2\cong L(\vx)$, and by Proposition \ref{orth} again,
$$\langle L_1,E\rangle=1\;\text{ and }\;\, F_{L_2L_1}^E=1.$$

Now suppose that $L_1\ncong L(\vw), L(\vx).$ If $\langle L_1,L_2 \rangle<0$,
then $\Ext^1(L_1,L_2)\ne0$ and hence, $\Ext^1(L_2,L_1)=0$. Thus, $F_{L_2L_1}^E=0$.

Assume $\langle L_1,L_2 \rangle \ge0$. Then
$$1\le \langle L_1,L_2\rangle+\langle L_1,L_1\rangle=\langle L_1,L(\vw)\rangle+\langle L_1,L(\vx)\rangle.$$
This implies that $\Hom (L_1,L(\vw))\ne0$ or $\Hom (L_1,L(\vx))\ne0$.
In the case when $\Hom(L_1,L(\vw))\ne0$, let $L_1=L(\vw-\vy), L_2=L(\vx+\vy)$ with
$\vy=\sum_{i=1}^3l'_i\vx_i+l'\vc>0$ and $\vx=\sum_{i=1}^{3}l_i\vx_i$ with
$0\leq l_i\leq p_i-2$ for $1\leq i\leq 3$. Then $l'_i=0$ or $l'_i=p_i-1-l_i$.
It follows that $\vx+\vy-\vw\ge0$ and so $\Hom(L_1,L(\vx))\ne0.$ Therefore,
there exists
$\vz=\sum_{i=1}^3(l_i+1)\vx_i+l\vc$ such that $L_1=L(\vx-\vz), L_2=L(\vw+\vz)$.
 The case $\Hom(L_1,L(\vx))\not=0$ can be treated similarly.

As conclusion, there is a subset $J\subseteq\{1,2,3\}$ such that
$$\aligned
& L_1=L\big(\sum_{j\notin J}l_j \vx_j+\sum_{i\in J}(p_i-1)\vx_i-(l+|J|)\vc\big),     \\
& L_2=L\big(\sum_{j\notin J}(p_j-1) \vx_j+\sum_{i\in J}l_i \vx_i+(l+|J|-2)\vc\big),
\endaligned$$
 and $\langle L_1, L_2\rangle=\dim \Hom (L_1,L_2)=2l+|J|-1$.

We claim that $F_{S_1L(\vw)}^{L_2}F_{S_1L_1}^{L(\vx)}\ne0$ if
and only if $ S_1\in \mathscr{S}_l^{\{\sigma_i\}_{i\in J}}$. For convenience,
assume $L=\co$ and denote by $\sigma_i=[S_{i,p_i-1}^{(p_i-1-l_i)}]$.
Then $[L_2]-[L(\vw)]=\sum_{i\in J}\sigma_i+l\delta$. For
$\varphi=\prod_{i\in J} {\rm x}_i^{p_i-1-l_i}\cdot f_l:L(\vw)\rightarrow L_2$ or $L_1\rightarrow L(\vx)$,
where $f_l\in\mathbf{k}[T_1,T_2]$ is a homogeneous polynomial of degree $l\vc$ (Note that the
coefficients of $\vx_j\, (j\notin J)$ in $L_2$ and $L(\vx)$ are different).
Hence, no direct summand of $\Coker \varphi$ lies in $\mathcal{T}_j$.

Similarly, $F_{S_2L(\vx)}^{L_2}F_{S_2L_1}^{L(\vw)}\ne0$ if and only if
$ S_2\in \mathscr{S}_{l+|J|-2}^{\{\sigma_j\}_{j\notin J}}$. Then by \eqref{Hall-number-ext-bdl-2},
we obtain that
\begin{align*}
    F_{L_2L_1}^E=\Psi_1-\Psi_2=
    s_l^{(|J|)}(q)-s^{(3-|J|)}_{l+|J|-2}(q)=f_{2l+|J|-1}(q)=f_{\langle L_1,L_2\rangle}(q),
\end{align*}
where both Proposition \ref{ssprop} and the equality $\langle L_1,L_2 \rangle=2l+|J|-1$
are used.
\end{proof}

Under the assumption $[L_1]+[L_2]=[E]$, we have
$$\langle L_1,E\rangle=\langle L_1,L_1\rangle+\langle L_1,L_2\rangle=1+\langle L_1,L_2\rangle$$
and
$$\langle E,E\rangle=1=\langle L_1,L_1\rangle+\langle L_1,L_2\rangle
+\langle L_2,L_1\rangle+\langle L_2,L_2\rangle.$$
Hence,
$$\langle L_1,L_2\rangle+\langle L_2,L_1\rangle=-1.$$
As a consequence,
$$\langle L_1,L_2\rangle \ge0 \Longleftrightarrow \langle L_2,L_1\rangle \le-1
\Longleftrightarrow\langle L_1,E\rangle \ge1.$$

\subsection{Hall polynomials $F_{SE'}^E$ for $S$ lying in a homogeneous tube}

In this subsection, we consider Hall polynomials $F_{SE'}^E$, where $E$ and $E'$ are
extension bundles and $S$ is an indecomposable torsion sheaf
taken from a homogeneous tube $\mathcal{T}_z$.

\begin{proposition} \label{hallpoly3}
Let $E, E'$ be extension bundles and $S$ be an indecomposable torsion sheaf taken from a
homogenous tube $\cT_z$ with $\deg z=d$. Assume $[E']+[S]=[E]$ and $S=S_z^{(n)}$ for some $n\geq 1$.
Then
\[  F_{SE'}^E=f_{dn}-f_{dn-1}+(q^d-1)\sum_{t\ge1}q^{d(t-1)}\big(f_{d(n-2t)}-f_{d(n-2t)-1}\big). \]
\end{proposition}

\begin{proof}
In view of \eqref{Hall-number-ext-bdl-1} and \eqref{Hall-number-ext-bdl-2}, we calculate
the corresponding values $\Psi_1$ and $\Psi_2$. Write $E=E_L\langle \vx\rangle$ for some
line bundle $L$ and $\vx\in\bbl$. Consider the following diagram with each row and
column being short exact sequences:
\[
\begin{tikzpicture}
\node (0) at (0,0) {$X_3$};
\node (1) at (1.5,0) {$L(\vx)$};
\node (2) at (3,0){$X_4$};
\node (01) at (0,1) {$E'$};
\node (11) at (1.5,1) {$W$};
\node (21) at (3,1){$S$};
\node (02) at (0,2) {$X_1$};
\node (12) at (1.5,2) {$L(\vw)$};
\node (22) at (3,2){$X_2$};	
\draw[->] (01) --node[above ]{} (11);
\draw[->] (11) --node[above ]{} (21);
\draw[->] (12) --node[above ]{} (11);
\draw[->] (11) --node[above ]{} (1);	
\draw[->] (0) --node[above ]{} (1);
\draw[->] (1) --node[above ]{} (2);
\draw[->] (02) --node[above ]{} (12);
\draw[->] (12) --node[above ]{} (22);
\draw[<-] (0) --node[above ]{} (01);
\draw[<-] (01) --node[above ]{} (02);
\draw[->] (21) --node[above ]{} (2);
\draw[->] (22) --node[above ]{} (21);	
\end{tikzpicture} \]
Then $W\cong E$ or $W\cong L(\vx)\oplus L(\vw)$. Furthermore,
$X_2,X_4$ are indecomposable torsion sheaves or zero and $X_1,X_3$ are line bundles.
If $X_2=0$, then $X_1=L(\vw),X_3=L(\vx-dn\vc)$, but $\Ext^1(L(\vx-dn\vc),L(\vw))=0$.
This contradicts that fact that $E'$ is an extension bundle. Therefore, $X_2\ne0$ and
$$\Psi_1=(q-1)\sum_{[X_1],[X_2],[X_3],[X_4]}q^{-\langle X_4,X_1\rangle}F_{X_2X_1}^{L(\vw)}
F_{X_3X_1}^{E'}F_{X_4X_2}^{S}F_{X_4X_3}^{L(\vx)}
\frac{a_{X_1}a_{X_2}a_{X_3}a_{X_4}}{a_{E'}a_{S}a_{L(\vw)}a_{L(\vx)}}.$$
We focus on non-zero terms on the right-hand side. Then $X_4=S_z^{(t)}$ for some $0\leq t<n$,
and hence,
$$X_2\cong S_z^{(n-t)},\,X_1\cong L(\vw-d(n-t)\vc),\; \text{ and }\;X_3\cong L(\vx-dt\vc).$$
By Theorem \ref{hallpoly2},
$$F_{L(\vx-dt\vc)L(\vw-d(n-t)\vc)}^{E'}=f_{\langle L(\vw-d(n-t)\vc),L(\vx-dt\vc)\rangle}=f_{d(n-2t)}.$$
Since $\langle X_4,X_1\rangle=-dt$ and
\[\frac{a_{S_z^{(n-t)}}a_{S_z^{(t)}}}{a_{S_z^{(n)}}}=
\begin{cases}
\dfrac{q^d-1}{q^d}, &\text{  if } 1\le t< n;  \\
1,  &\text{  if } t=0, \end{cases} \]
we obtain that
\begin{align}
\label{green3}
\Psi_1=f_{dn}+\frac{q^d-1}{q^d}\sum_{t\ge1}q^{dt}f_{d(n-2t)}.
\end{align}

To calculate $\Psi_2$, consider similarly the following commutative diagram:
\[
\begin{tikzpicture}
\node (0) at (0,0) {$X_3$};
\node (1) at (1.5,0) {$L(\vw)$};
\node (2) at (3,0){$X_4$};
\node (01) at (0,1) {$E'$};
\node (11) at (1.5,1) {$W$};
\node (21) at (3,1){$S$};
\node (02) at (0,2) {$X_1$};
\node (12) at (1.5,2) {$L(\vx)$};
\node (22) at (3,2){$X_2$};	
\draw[->] (01) --node[above ]{} (11);
\draw[->] (11) --node[above ]{} (21);
\draw[->] (12) --node[above ]{} (11);
\draw[->] (11) --node[above ]{} (1);	
\draw[->] (0) --node[above ]{} (1);
\draw[->] (1) --node[above ]{} (2);
\draw[->] (02) --node[above ]{} (12);
\draw[->] (12) --node[above ]{} (22);
\draw[<-] (0) --node[above ]{} (01);
\draw[<-] (01) --node[above ]{} (02);
\draw[->] (21) --node[above ]{} (2);
\draw[->] (22) --node[above ]{} (21);	
\end{tikzpicture} \]
Then $W\cong L(\vx)\oplus L(\vw)$ and $X_2\ne0$ as above. Therefore,
\begin{align*}
\Psi_2=(q-1)\sum_{[X_1],[X_2],[X_3],[X_4]}q^{-\langle X_4,X_1\rangle}F_{X_2X_1}^{L(\vx)}F_{X_3X_1}^{E'}F_{X_4X_2}^{S}F_{X_4X_3}^{L(\vw)}
\frac{a_{X_1}a_{X_2}a_{X_3}a_{X_4}}{a_{E'}a_{S}a_{L(\vw)}a_{L(\vx)}}.
\end{align*}
By considering non-zero terms in the right-hand side, we have $X_4=S_z^{(t)}$ for some
$0\leq t<n$, and then
$$X_2\cong S_z^{(n-t)},\, X_1\cong L(\vx-d(n-t)\vc),\;\text{ and }\; X_3\cong L(\vw-dt\vc).$$
Again, by Theorem \ref{hallpoly2},
$$F_{L(\vw-dt\vc)L(\vx-d(n-t)\vc)}^{E'}=
f_{\langle L(\vx-d(n-t)\vc),L(\vw-dt\vc)\rangle}=f_{d(n-2t)-1}.$$
From the equality $\langle X_4,X_1\rangle=-dt$ it follows that
\begin{align}
\label{green4}
\Psi_2=f_{dn-1}+\frac{q^d-1}{q^d}\sum_{t\ge1}q^{dt}f_{d(n-2t)-1}.
\end{align}
Combining the formulas \eqref{green3} and \eqref{green4}, we obtain
\[ F_{SE'}^E=\Psi_1-\Psi_2=f_{dn}-f_{dn-1}+(q^d-1)\sum_{t\ge1}q^{d(t-1)}(f_{d(n-2t)}-f_{d(n-2t)-1}). \]
\end{proof}

\subsection{Hall polynomials $F_{SE'}^E$ for $S$ lying in an exceptional tube}

Let $E,E'$ be two extension bundles with $E=E_L\langle \vx\rangle$
for a line bundle $L$ and some $\vx\in\bbl$.
Set $N=\lfloor\frac{1}{2}\langle E',E\rangle\rfloor-1$, and write $\vx-\vw=\sum\limits_{k=1}^{3}l_k\vx_k-\vc$ with
$1\leq l_k\leq p_k-1$ for $1\leq k\leq 3$.
Let $S=S_{i,j}^{(n)}$ be an indecomposable torsion sheaf taken from an exceptional
tube $\mathcal{T}_i$ for a fixed $i\in I$.

\begin{lemma} \label{preparing lemma}
Keep notations as above. Assume $[E']+[S]=[E]$.
\begin{itemize}
  \item[(1)] If $\Hom(L(\vx),\mathrm{top\,} S)\ne0$, then $n\not \equiv l_i\,(\mod p_i)$, and $N=\lfloor\frac{n-l_i}{p_i}\rfloor$.
  \item[(2)] If $\Hom(L(\vw),\mathrm{top\,} S)\ne0$, then $n\not \equiv p_i-l_i\,(\mod p_i)$, and $N=\lfloor\frac{n-(p_i-l_i)}{p_i}\rfloor$.
\end{itemize}\end{lemma}

\begin{proof} We only prove assertion (1). The proof of (2) is similar.

We first claim that $n\not \equiv l_i\,(\mod p_i)$.
Otherwise, suppose $n=l_i+tp_i$ for some $t\geq 0$. Then
$[L(\vx)]-[S]=[L(\vx-l_i\vx_i)]-t\delta$, and hence,
$$[E']=[E]-[S]=[L(\vw)]+[L(\vx)]-[S]=[L(\vw)]+[L(\vx-l_i\vx_i)]-t\delta.$$ Note that
$\langle L(\vw),L(\vx-l_i\vx_i)\rangle=0=\langle L(\vx-l_i\vx_i),L(\vw) \rangle$. Therefore,
\begin{align*} \langle E',E'\rangle&=\langle [L(\vw)]+[L(\vx-l_i\vx_i)]-t\delta, [L(\vw)]+[L(\vx-l_i\vx_i)]-t\delta\rangle=2.\end{align*}
This contradicts the fact that $E'$ is an exceptional object.

Further, we have
\begin{align*} \langle E',E\rangle&=\langle [L(\vw)]+[L(\vx)]-[S], [L(\vw)]+[L(\vx)]\rangle\\
&=1-\langle [S], [L(\vw)]+[L(\vx)]\rangle\\
&=1+\dim\Ext^{1}(S, L(\vw))+\dim\Ext^{1}(S, L(\vx))\\
&=1+\dim\Ext^{1}(S_{i,j}^{(n)}, S_{i,j-l_i}^{(n)})+\dim\Ext^{1}(S_{i,j}^{(n)}, S_{i,j}^{(n)})\\
&=1+\lfloor\frac{n}{p_i}\rfloor+\lfloor\frac{n-l_i}{p_i}\rfloor+\lceil\frac{l_i}{p_i}\rceil\\
&=2+\lfloor\frac{n}{p_i}\rfloor+\lfloor\frac{n-l_i}{p_i}\rfloor.\end{align*}
This implies that
$$N=\lfloor\frac{1}{2}\langle E',E\rangle\rfloor-1=1+\lfloor\frac{n-l_i}{p_i}\rfloor-1
=\lfloor\frac{n-l_i}{p_i}\rfloor.$$
\end{proof}

\begin{proposition} \label{hallpoly4}
Let $E,E'$ be extension bundles and $S=S_{i,j}^{(n)}$ be an indecomposable torsion
sheaf taken from an exceptional tube as above.
Assume $[E']+[S]=[E]$ and $\Hom(E,\mathrm{top}\, S)\ne0$.
Then
 \begin{align*}		
  F_{SE'}^E=\begin{cases}
			1, & \text{ if } N=-1;\\
f_{N+1}-f_{N}+(-1)^{N}, & \text{ if }  N\geq 0,
		\end{cases}
	\end{align*}
where $N=\lfloor\frac{1}{2}\langle E',E\rangle\rfloor-1$.
\end{proposition}

\begin{proof} In view of \eqref{Hall-number-ext-bdl-1}, we now
calculate $F_{SE'}^E=\Psi_1-\Psi_2$. Since $\Hom(E,\mathrm{top}\, S)\ne0$,
we have
$$\Hom(L(\vx),\mathrm{top\,} S)\ne0\, \text{ or }\;\, \Hom(L(\vw),\mathrm{top}\, S)\neq 0.$$
We only consider the first case, and the second one can be treated similarly.

Without loss of generality, we may take $i=1$, i.e., $S=S_{1,j}^{(n)}$, and write
$$\vx-\vw=\sum\limits_{k=1}^{3}l_k\vx_k-\vc$$
with $1\leq l_k\leq p_k-1$ for $1\leq k\leq 3$. In this case,
 $\Hom(L(\vx),S_{1,j})\ne0$ and $\Hom(L(\vw),S_{1,j-l_1})\ne0$.

We first calculate the non-zero terms of
$$\Psi_1:=\frac{1}{q-1}\sum_{[X_1],[X_2],[X_3],[X_4]}q^{-\langle X_4,X_1\rangle}F_{X_3X_1}^{E'}F_{X_4X_2}^{S}\frac{a_{X_1}a_{X_2}a_{X_3}a_{X_4}}{a_{E'}a_{S}}
F_{X_2X_1}^{L(\vw)}F_{X_4X_3}^{L(\vx)}.$$
If $X_2=0$, then $X_4=S_{1,j}^{(n)}$, $X_1=L(\vw)$ and $X_3=L(\vx-n\vx_1)$. In this case,
\begin{align*}
n_1:&=-\langle X_{4},X_1\rangle=-\langle S_{1,j}^{(n)},L(\vw)\rangle=\dim\Ext^{1}(S_{1,j}^{(n)},L(\vw))\\
&=\dim\Ext^{1}(S_{1,j}^{(n)},S_{1,j-l_1}^{(n)})=\lfloor\frac{n-l_1}{p_1}\rfloor+1=N+1, \text{\;\;
(by\;\; Lemma \ref{preparing lemma})}\\
n_2:&=\langle X_1,X_3\rangle=\langle [L(\vw)],[L(\vx)]- [S_{1,{j}}^{(n)}]\rangle
=-\langle [L(\vw)], [S_{1,{j}}^{(n)}]\rangle\\
&=-\dim\Hom(L(\vw),S_{1,j}^{(n)})=-\dim\Hom(S_{1,j-l_1}^{(n)},S_{1,j}^{(n)})
=-\lceil\frac{n-l_1}{p_1}\rceil,
\end{align*}
and
$$\frac{a_{X_1}a_{X_2}a_{X_3}a_{X_4}}{a_{E'}a_{S}}
=\frac{a_{S}a_{L(\vw)}a_{L(\vx-n\vx_1)}}{a_{E'}a_{S}}=q-1.$$
If $X_2\neq 0$, then $X_4=S_{1,j}^{(l)}$ and $X_2=S_{1,j-l}^{(n-l)}$ for some
$0\leq l\leq n-1$. Hence, $\Hom(L(\vw), \text{top\,} X_2)\neq 0$, which implies
$j-l\equiv j-l_1\,(\mod p_1)$, and thus, $l=l_1+tp_1$ for some $t\geq 0$.
In this case, $X_1=L(\vw-(n-l)\vx_1)$ and $X_3=L(\vx-l\vx_1)$, and
\begin{align*}
n_3:&=-\langle X_{4},X_1\rangle=-\langle S_{1,j}^{(l)},L(\vw-(n-l)\vx_1)\rangle
\\&=t-\langle [S_{1,j}^{(l_1)}],[L(\vw)]-[S_{1,j}^{(n)}]+[S_{1,j}^{(l)}]\rangle\\
&=t-(-1)+(\lceil\frac{n}{p_1}\rceil-\lceil\frac{n-l_1}{p_1}\rceil-\lfloor\frac{l_1}{p_1} \rfloor)
-1\\
&=t-\lceil\frac{n-l_1}{p_1} \rceil+\lceil\frac{n}{p_1} \rceil,\\
n_4:&=\langle X_1,X_3\rangle=\langle [L(\vw-(n-l)\vx_1)],[L(\vx-l\vx_1)]\rangle\\
&=\langle [L(\vw)]-[S_{1,{j-l}}^{(n-l)}],{[L(\vx)]}-[S_{1,{j}}^{(l)}]\rangle\\
&=\langle [L(\vw)],{[L(\vx)]}\rangle-\langle [L(\vw)],[S_{1,{j}}^{(l)}]\rangle-\langle [S_{1,{j-l}}^{(n-l)}],{[L(\vx)]}\rangle+\langle [S_{1,{j-l}}^{(n-l)}],[S_{1,{j}}^{(l)}]\rangle\\
&=0-{\dim\Hom(L(\vw), S_{1,{j}}^{(l)})}+{\dim\Ext^1(S_{1,{j-l}}^{(n-l)},L(\vx))}+\langle [S_{1,{j-l}}^{(n-l)}],[S_{1,{j}}^{(l)}]\rangle\\
&=-{\dim\Hom(S_{1,j-l_1}^{(l)}, S_{1,{j}}^{(l)})}+{\dim\Ext^1(S_{1,{j-l}}^{(n-l)},S_{1,{j}}^{(n-l)})}
+\langle [S_{1,{j-l}}^{(n-l)}],[S_{1,{j}}^{(l)}]\rangle\\
 &=-t+\big(\lfloor\frac{n}{p_1}\rfloor-t\big)+\big(\lfloor\frac{n-l_1}{p_1}\rfloor-\lfloor\frac{n}{p_1}\rfloor\big)\\
&=N-2t,
\end{align*}
and
\begin{align*}
\frac{a_{X_1}a_{X_2}a_{X_3}a_{X_4}}{a_{E'}a_{S}}&=
\frac{(q-1)^2}{q-1}\times \frac{a_{S_{1,j}^{(l)}}a_{S_{1,j-l}^{(n-l)}}}{a_{S_{1,j}^{(n)}}}
\\&=(q-1)\frac{(q-1)q^{\dim\End(S_{1,j}^{(l)})-1}\cdot (q-1)q^{\dim\End(S_{1,{j-l}}^{(n-l)})-1}}{(q-1)q^{\dim\End(S_{1,{j}}^{(n)})-1}}
\\&=(q-1)^2q^{-1}q^{\dim\End(S_{1,{j}}^{(l)})+\dim\End(S_{1,{j-l}}^{(n-l)})-\dim\End(S_{1,{j}}^{(n)})}
\\&=(q-1)^2q^{-1}q^{\lceil\frac{l}{p_1} \rceil+\lceil\frac{n-l}{p_1} \rceil-\lceil\frac{n}{p_1} \rceil}
\\&=(q-1)^2q^{\lceil\frac{n-l_1}{p_1} \rceil-\lceil\frac{n}{p_1} \rceil}.
\end{align*}
Since $0\leq t=\frac{l-l_1}{p_1}\leq \frac{n-l_1-1}{p_1}$, we conclude that
\begin{align}
\label{green5}
\Psi_1=&\frac{1}{q-1}q^{n_1}f_{n_2}\cdot(q-1)+
\frac{1}{q-1}\sum_{t\ge0}^{\lfloor\frac{n-l_1-1}{p_1}\rfloor}q^{-n_3}f_{n_4}\cdot (q-1)^2q^{\lceil\frac{n-l_1}{p_1} \rceil-\lceil\frac{n}{p_1} \rceil}
\\\notag =&q^{n_1}f_{n_2}+
\sum_{t\ge0}(q-1)q^{t}f_{n_1-1-2t}
\\\notag =&q^{n_1}f_{n_2}+f_{n_1}+(-1)^{n_1-1}.
\end{align}
Here the last equality follows from \eqref{sumf}.

Next we consider the non-zero terms of
$$\aligned
\Psi_2:=\frac{1}{q-1}\sum_{[X_1],[X_2],[X_3],[X_4]}q^{-\langle X_4,X_1\rangle}F_{X_3X_1}^{X}F_{X_4X_2}^{X'}
\frac{a_{X_1}a_{X_2}a_{X_3}a_{X_4}}{a_{X}a_{X'}}F_{X_2X_1}^{L(\vx)}F_{X_4X_3}^{L(\vw)}.
\endaligned$$
Since $\Ext^{1}(L(\vw), L(\vx))=0$ and $\Hom(L(\vw), S_{1,j})=0$, we have $X_4=0$. Thus,
$X_2=S=S_{1,{j}}^{(n)}$, $X_3=L(\vw)$ and $X_1=L(\vx-n\vx_1)$. In this case,
\begin{align} \label{green6}	\Psi_2=&\frac{1}{q-1}F_{SL(\vx-n\vx_1)}^{L(\vx)}F_{L(\vw)L(\vx-n\vx_1)}^{E'}
\frac{a_{S}a_{L(\vw)}a_{L(\vx-n\vx_1)}}{a_{E'}a_{S}}=f_{n_5},
\end{align}
where
\[	n_5=\langle [L(\vx)]- [S_{1,{j}}^{(n)}], [L(\vw)]\rangle=\lfloor\frac{n-l_1}{p_1}\rfloor=N.\]
	
By Proposition \ref{simplify} together with \eqref{green5} and \eqref{green6}, we have
\[	F_{SE'}^E=\Psi_1-\Psi_2=q^{n_1}f_{n_2}+f_{n_1}+(-1)^{n_1-1}-f_{n_5}.	\]
If $N=-1$, i.e., $1\le n< l_1$, then $n_1=n_2=0$, and hence,
 \begin{align*}
     F_{SE'}^E=f_{0}+f_{0}+(-1)-f_{-1}=1.
 \end{align*}
 If $N\geq 0$, then $n>l_1$ since $n\not \equiv l_i\;(\mod p_i)$ by Lemma \ref{preparing lemma}.
Therefore, $n_2<0$ and thus,
 \begin{align*}
      F_{SE'}^E&=f_{N+1}-f_{N}+(-1)^{N}.
 \end{align*}
This finishes the proof.
\end{proof}

\begin{remark} Keep notations as above. Since $F_{SE'}^E$ is
a polynomial in $q=|\bfk|$ of degree $N+1$, we have $F_{SE'}^E(q)\neq 0$ when
$q\gg0$. In this case, there exists a short exact sequence
\begin{equation}\label{exact} 0\lra E'\lra E\lra S\lra 0 \end{equation}
if and only if  $[E']+[S]=[E]$ and $\Hom(E,\mathrm{top}\, S)\ne0$.

However, this is not true in general. For example, let $\mathbb{X}$ be
the weighted projective line of weight type $(2,2,2)$. Let $S=S_{i,j}^{(2)}$ be an indecomposable torsion
sheaf taken from an exceptional tube and $E$ be the Auslander
bundle given by the non-split exact sequence
$$0\lra \co(\vw)\lra E\lra \co\lra 0.$$
Then we have
$\langle \tau E,E\rangle=3$, $[\tau E]+[S]=[E]$, and $\Hom(E,\mathrm{top}\, S)\ne0$. By Proposition \ref{hallpoly4} we obtain $N=0$ and
$F_{SE'}^E=f_1(q)=q-2$. Hence, if $\bfk=\bbF_2$, then $F_{SE'}^E=0$. Therefore,
there does not exist such an exact sequence of the form \eqref{exact}.
\end{remark}

\bigskip

\section{Derived Hall algebras and Hall polynomials for tame quivers}

In this section we recall from \cite{Toen,XX} the definition of the derived Hall algebra
of the bounded derived category of a finitary abelian category. In terms of the
isomorphism between derived Hall algebras of tame quivers and weighted projective lines of
domestic type, we apply the results in the previous sections to obtain Hall polynomials for
representations of tame quivers which are calculated in \cite{Sz0, SS2} by a case-by-case
analysis.

Let $\ca$ be a finitary abelian category over a finite field $\mathbf k$ and
$D^b(\mathcal{A})$ be the bounded derived category of $\mathcal{A}$. For any three
objects $X,Y,L$ in $D^b(\mathcal{A})$, let $\Hom_{D^b(\mathcal{A})}(X,Y)_L$ denote the
subset of $\Hom_{D^b(\mathcal{A})}(X,Y)$ consisting of morphisms $X \to Y$ whose cone is
isomorphic to $L$, and set
$$\{X, Y\}:=\prod_{i>0}|\Hom_{D^b(\mathcal{A})}(X[i],Y)|^{(-1)^i}.$$
For $X\in D^b(\mathcal{A})$, let $\widetilde{a}_{X}:=|\mathrm{Aut}_{D^b(\mathcal{A})}(X)|$
be the cardinality of the automorphism group. Recall from \cite{Toen,XX} that
the \emph{derived Hall number} associated with three objects $L,X,Y$ in $D^b(\mathcal{A})$ is defined as
\begin{align}\label{definition of derived Hall number}
G_{XY}^L:=\frac{|\Hom(L,X)_{Y[1]}|}{\widetilde{a}_X}\cdot\frac{\{L,X\}}{\{X,X\}}
=\frac{|\Hom(Y,L)_{X}|}{\widetilde{a}_Y}\cdot \frac{\{Y,L\}}{\{Y,Y\}},
\end{align}
which can be described by the following derived Riedtmann--Peng formula
(c.f., \cite[Prop. 3.3]{SCX}):
\[ G_{XY}^L=\frac{|\Hom(X,Y[1])_{L[1]}|\{X,Y[1]\}\{L,L\}}{\{X,X\}\{Y,Y\}}\cdot\frac{a_{L}}{a_{X}a_Y}.\]
The derived Hall algebra $\mathcal{DH}(\mathcal{A})$ of $\ca$
is by definition the $\mathbb{Q}$-vector space with the basis
$\{u_{[X]} \mid [X] \in \mathrm{Iso}(D^b(\mathcal{A}))\}$, endowed with multiplication defined by
\[u_{[X]}\cdot u_{[Y]}=\sum_{[L]\in{\rm{Iso}}(D^b(\mathcal{A}))}G_{XY}^L\, u_{[L]}. \]
It is an associative algebra, that is, for
given objects $X,Y,Z,M$ in $D^b(\mathcal{A})$, we have
\begin{align}\label{associativity of derived Hall number}
\sum_{[L]\in{\rm{Iso}}(D^b(\mathcal{A}))} G_{XY}^{L}G_{LZ}^{M}
 =\sum_{[L]\in{\rm{Iso}}(D^b(\mathcal{A}))} G_{XL}^{M}G_{YZ}^{L},
\end{align}
where both sums are taken over the set ${\rm{Iso}}(D^b(\mathcal{A}))$ of isomorphism
classes of objects in $D^b(\mathcal{A})$.

Derived Hall algebras are closely related to Hall algebras. Indeed, we have
the following embedding of algebras.

\begin{lemma} Let $\mathcal{A}$ be a hereditary abelian category. Then there exists
an algebra embedding
\[   \varphi:\mathcal{H}(\mathcal{A})\longrightarrow \mathcal{DH}(\mathcal{A}),     \]
taking
\[ \qquad \qquad   [X]\longmapsto \{X,X\}\cdot u_{[X]}.     \]
\end{lemma}

\begin{proof} This follows directly from the definitions and the equalities $$|\Hom_{D^b(\mathcal{A})}(X,Y[1])_{L[1]}|=|\Ext_\mathcal{A}^1(X,Y)_L|,\;\,\forall\, X, Y,L\in {\mathcal A}.$$
\end{proof}

It is well-known that for each tame quiver $Q$ without oriented cycles, there is
a weighted projective line $\X=\X_Q$ of domestic type such that the module category
$\bfk Q\mbox{-mod}$ of finite dimensional left modules over the path algebra
$\bfk Q$ and $\coh\X$ are derived equivalent. Note that
the correspondence between the types of $Q$ and $\X$ is given as follows
\[\begin{tabular}{|c|c|c|c|c|c|}
		\hline
 &&&&&\\
$Q$	&$\widetilde{\mathbb{A}}_{p,q}$&
 $\widetilde{\mathbb{D}}_{n}$&
 $\widetilde{\mathbb{E}}_{6}$ & $\widetilde{\mathbb{E}}_{7}$ &
 $\widetilde{\mathbb{E}}_{8}$\\
		\hline
 &&&&&\\
$\X$ &$(p,q)$&	$(2,2,n-2)$	 &$(2,3,3)$ &$(2,3,4)$  & $(2,3,5)$\\\hline
\end{tabular}\]	

\bigskip

From now onwards, let $Q$ be a tame quiver without oriented cycles and
$\mathrm{K}_0(\bfk Q)$ denote the Grothendieck group of $\bfk Q\mbox{-mod}$.
By \cite{DR}, the module category $\bfk Q\mbox{-mod}$ admits full subcategories $\mathscr{P,R}$ and
$\mathscr{I}$ which consist of preprojective, regular and preinjective $\bfk Q$-modules,
respectively. Then under the canonical derived equivalence $D^b(\bfk Q\mbox{-mod})\cong D^b(\coh\X)$,
$\mathscr{P,I}$ correspond to the vector bundles or their shifts, and $\mathscr{R}$ corresponds to
the torsion sheaves in $\coh\X$. We will simply identify $\mathscr{R}$ with $\coh_0$-$\X$
and use the same notation to denote the tubes in $\mathscr{R}$ as in $\coh_0$-$\X$.
In particular, for an indecomposable regular $\bfk Q$-module $R$ in a tube $\cT$, by $\mbox{top\,}R$ we
denote its relative top in $\cT$.

By \cite[Prop. 5]{Cr10}), there is an isomorphism
$\mathcal{DH}(\mathcal{\coh \X}) \cong \mathcal{DH}(\bfk Q\mbox{-mod})$ of derived
Hall algebras. Thus, we can reformulate the results in Sections 3 and 4 in terms
of representations of $Q$. This indeed gives several results on Hall polynomials
obtained in \cite{Sz2,SS2}; see the following propositions. Since their proofs are very similar,
we omit most of them.

\begin{proposition} %\cite[Theorem 4.1]{Sz2}
 Assume that $P,P'$ are both indecomposable modules of defect $-1$ in $\mathscr{P}$ and
$R=\bigoplus_k R_k$ is a module in $\mathscr{R}$ with indecomposable summands $R_k$ from
pairwise distinct tubes. Then there exists a short exact sequence
\[ 	0\lra P\lra P'\lra R\lra 0 	\]
if and only if $[P]+[R]=[P']$ in $\mathrm{K}_0(\bfk Q)$ and $\Hom(P',\mathrm{top\,} R_k)\ne0$.
In this case, we have
\[	 F_{RP}^{P'}=1. 	 \]
\end{proposition}

\begin{proof} Consider the canonical derived equivalence by
$$G:\mathcal{DH}(\bfk Q\mbox{-mod})\longrightarrow \mathcal{DH}(\mathcal{\coh \X}).$$
Then $G(P),G(P')$ are line bundles with $G([P])+G([R])=G([P'])$, and $G(R)=\bigoplus_k G(R_k)$
satisfies $\Hom(G(P'),\mathrm{top\,} G(R_k))\ne0$.  By Proposition \ref{hallpoly1},
we have $F_{G(R)G(P)}^{G(P')}=1$ and thus, $F_{RP}^{P'}=1$.
\end{proof}

\begin{proposition} %\cite[Thm 4.1]{Sz2}
Assume that $P,P'$ are both indecomposable modules of defect $-1$ in $\mathscr{P}$
and $R$ is an indecomposable module in $\mathscr{R}$. Then there exist a module $R'$
(may be zero) in $\mathscr{R}$ and a short exact sequence
\[ 0\lra P\lra P'\oplus R'\lra R\lra 0 \]
if and only if $R'$ is a submodule of $R$ such that there is a short exact sequence
\[ 	0\lra P\lra P'\lra R/R'\lra 0. 	\]
In this case, we have
\[  F_{RP}^{P'\oplus R'}=\begin{cases}
        |\Hom(P,R)|,& \text{ if } R\cong R';\\
        a_{R'},& \text{ if } R\ncong R'.
     \end{cases}     \]
\end{proposition}

The two propositions above are a supplement of \cite[Thm 4.1]{Sz2}
where $R$ is a direct sum of indecomposable representations taken from a fixed
homogeneous tube. The following result is a reformulation of Proposition \ref{hallpoly1''}.

\begin{proposition} Assume that $P,P'$ are both indecomposable modules of defect
$-1$ in $\mathscr{P}$ and $R,R',R''$ are indecomposable modules in $\mathscr{R}$
(or zero modules). Then there exists a short exact sequence
 \[ 	0\lra P\oplus R\lra P'\oplus R'\lra R''\lra 0 	\]
 if and only if there are modules $R_1,R_2$ and short exact sequences
 \[  0\to R\to R'\to R_1\to 0,\quad 0\to P\to P'\to R_2\to 0,
 \quad    0\to R_1\to R''\to R_2\to 0.\]
    In this case, we have
 \[ F_{R''P\oplus R}^{P'\oplus R'}=\begin{cases}
        |\Hom(P,R'')|,& \text{ if } P\cong P';\\
        a_{R/R'},& \text{ if } P\ncong P'.
     \end{cases}     \]
\end{proposition}

The following three propositions are reformulations of Theorem \ref{hallpoly2},
Proposition \ref{hallpoly3}, and Proposition \ref{hallpoly4}, respectively.

\begin{proposition} [{\rm\cite[Thm. 7.1]{SS2}}]
Assume that $P_1,P_2$ are indecomposable modules of defect $-1$ in $\mathscr{P}$ and
$P$ is an indecomposable module of defect $-2$ in $\mathscr{P}$. Assume
$[P_1]+[P_2]=[P]$ in $\mathrm{K}_0(\bfk Q)$ and the Euler form
$\langle P_1,P_2\rangle=n\ge0$. Then
\[ F_{P_2P_1}^P=f_n(q). \]
\end{proposition}

\begin{proposition} [{\rm\cite[Thm. 7.3]{SS2}}]
 Assume that $P,P'$ are indecomposable modules of defect $-2$ in
 $\mathcal{P}$ and $R$ is an indecomposable module in $\mathscr{R}$
 of class $dn\delta$ taken from a homogeneous tube $\mathcal{T}_z$ with
 $\mathrm{deg}\ z=d$. Assume $[P']+[R]=[P]$ in $\mathrm{K}_0(\bfk Q)$. Then
 \[ F_{RP'}^P=f_{dn}-f_{dn-1}+(q^d-1)\sum_{t\ge1}q^{d(t-1)}(f_{d(n-2t)}-f_{d(n-2t)-1}).	\]
\end{proposition}

\begin{proposition} [{\rm \cite[Thm. 7.5]{SS2}}]
 Assume that $P,P'$ are indecomposable modules of defect $-2$
in $\mathscr{P}$ and $R$ is an indecomposable module in $\mathscr{R}$ taken
from an exceptional tube. Assume $[P']+[R]=[P]$ and $\Hom(P,\mathrm{top\,} S)\ne0$,
and set $N=\lfloor\frac{1}{2}\langle P',P\rangle\rfloor-1.$
Then
  \begin{align*}
     F_{RP'}^P=\begin{cases}
        1,& \text{ if }N=-1;\\
        f_{N+1}-f_{N}+(-1)^{N}, & \text{ if }N\ge0.     \end{cases}
 \end{align*}
\end{proposition}

\begin{proof} We only need to deal with the Euler form. In the first case in
Proposition \ref{hallpoly4}, we have
$$\langle E',E\rangle=\langle E,E\rangle-\langle S,E\rangle
=1-\langle S,L(\vw)\rangle-\langle S,L(\vx)\rangle,$$
where $-\langle S,L(\vw)\rangle=\lfloor\frac{n-l_i}{p_i}\rfloor+1$ and
$-\langle S,L(\vw)\rangle-1\le -\langle S,L(\vx)\rangle\le -\langle S,L(\vw)\rangle$.
Hence, $\lfloor\frac{1}{2}\langle P',P\rangle\rfloor=\lfloor\frac{n-l_i}{p_i}\rfloor+1$.
The second case can be treated similarly.
Therefore, we can transform the Hall polynomials in Proposition \ref{hallpoly4} to those
in $\bfk Q\mbox{-mod}$.
\end{proof}

To obtain the Hall polynomials involving preinjective $\bfk Q$-modules, we apply the
following result from \cite{Ruan}, which indicates the relationship between derived
Hall numbers under rotation of triangles.

\begin{lemma} [{\rm \cite[Lem.~2.1]{Ruan}}] \label{rotation}
Let $\ca$ be a finitary abelian category over $\mathbf k$. Then for $X,Y,L\in D^b(\mathcal{A})$, we have
\begin{equation}\label{rotation in Db(A)}
\frac{G_{XY}^L}{a_L\cdot \{L,L\}}=\frac{G_{Y[1]L}^X}{a_X\cdot\{X,X\}}
=\frac{G_{L X[-1]}^Y}{a_Y\cdot \{Y,Y\}}.
\end{equation}
In particular, if $X,Y,L\in\mathcal{A}$, then
\begin{equation}\label{rotation in A}\frac{G_{XY}^L}{a_L}=\frac{G_{Y[1]L}^X}{a_X}
=\frac{G_{LX[-1]}^Y}{a_Y}.
\end{equation}
  \end{lemma}

Let $P$ (resp., $\widetilde{P}$) be an indecomposable module in $\mathscr{P}$ of defect $-1$
(resp., $-2$); $I, I_1, I_2$ (resp., $\widetilde{I}$) indecomposable modules in $\mathscr{I}$
of defect $1$ (resp., $2$); $R, R_1, R_2$ indecomposable regular modules; and $R^h$ (resp., $ R^e$)
an indecomposable regular module from a homogeneous (resp., an exceptional) tube.

By Lemma \ref{rotation} together with Propositions 5.2--5.7 and an easy calculations of
automorphism groups, we obtain the following formulas of Hall polynomials, under the assumption
that the corresponding exact sequences exist.

\begin{proposition} [{\rm \cite[Thms. 7.2, 7.4\, \&\, 7.6]{SS2}}]
 Keep the notations above. Then
\begin{align*}
 F_{IP}^R&=\frac{a_R}{q-1},\qquad F_{I_1R}^{I_2}=1,\\
 F_{IP\oplus R_1}^{ R_2}&=\frac{a_{R_1}a_{R_2}}{a_{P\oplus R_1}}=\frac{a_{R_2}}{(q-1)|\Hom(P,R_1)|},\\
 F_{I\widetilde{P}}^P&=f_{n-1}(q), \text{ for } n=\langle \widetilde{P},P\rangle\ge1,\\
 F_{\widetilde{I}\widetilde{P}}^{R^h}&=\frac{a_{R^h}}{q-1}(f_{dn}-f_{dn-1}+(q^d-1)
 \sum_{t\ge1}q^{d(t-1)}(f_{d(n-2t)}-f_{d(n-2t)-1})),
 \end{align*}
where $R^h$ is of class $dn\delta$ taken from a homogeneous tube $\mathcal{T}_z$
with $\mathrm{deg}\ z=d$, and
\begin{align*}
F_{\widetilde{I}\widetilde{P}}^{R^e}=&\begin{cases}
        \frac{a_{R^e}}{q-1}, & \text{ if }{N=-1};\\
          \frac{a_{R^e}}{q-1}(f_{N+1}-f_{N}+(-1)^{N}),& \text{ if }N\ge0,
\end{cases}
 \end{align*}
for $N=\lfloor-\frac{1}{2}\langle \widetilde{I},\widetilde{P}\rangle\rfloor+1$.
\end{proposition}

\bigskip

\noindent {\bf Acknowledgements.}\quad
Bangming Deng was supported by the National Natural Science Foundation of China
(Grant No. 1203100). Shiquan Ruan was partially supported by the Natural Science
Foundation of Xiamen (Grant No. 3502Z20227184),
Fujian Provincial Natural Science Foundation of China (Grant Nos. 2024J010006 and 2022J01034),
the National Natural Science Foundation of China (Grant No. 12271448), and the Fundamental
Research Funds for Central Universities of China (Grant No. 20720220043).

\end{document}